\documentclass{amsart}
\usepackage{amsmath,amsfonts,amsthm,amscd, amssymb, latexsym}
\newtheorem{thm}{Theorem}[section]
\newtheorem{prop}[thm]{Proposition}
\newtheorem{cor}[thm]{Corollary}
\newtheorem{lem}[thm]{Lemma}
\newtheorem{claim}[thm]{Claim}
\newtheorem{rem}[thm]{Remark}
\newtheorem{defn}[thm]{Definition}

\newcommand{\pikm}{\pi_{k,m}}
\newcommand{\RR}{\mathbb{R}}
\newcommand{\ZZ}{\mathbb{Z}}

\newcommand{\CC}{\mathbb{C}}
\newcommand{\Ztwo}{\mathbb{Z}/2\mathbb{Z}}

\newcommand{\FF}{\mathbb{F}}
\newcommand{\MsigX}{ M \times_{\Sigma_k} X^{(k)}}
\newcommand{\MtsigX}{ M \rtimes_{\Sigma_k} X^{(k)}}
\newcommand{\MsigSd}{ M \times_{\Sigma_k} (S^d)^{(k)}}
\newcommand{\MtsigSd}{ M \rtimes_{\Sigma_k} (S^d)^{(k)}}

\begin{document}
\title{Configuration spaces and braid groups}

\author{Fred Cohen and Jonathan Pakianathan}

\begin{abstract}
The main thrust of these notes is 3-fold:
(1) An analysis of certain $K(\pi,1)$'s that arise from the connections
between
configuration spaces, braid groups, and mapping class groups,
(2) a function space interpretation of these results, and
(3) a homological analysis of the
cohomology of some of these groups for genus zero, one, and two
surfaces possibly with marked points, 
as well as the cohomology of certain associated function
spaces.

An example of the type of results given here is an analysis of
the space k particles moving on a punctured torus up to
equivalence by the natural $SL(2,\mathbb{Z})$ action.

\end{abstract}

\maketitle

\section{Introduction}

The main thrust of these notes is 3-fold: \\
\noindent
(1) An analysis of certain $K(\pi,1)$-spaces that arise from the connections
between
configuration spaces, braid groups, and mapping class groups. \\
\noindent
(2) A function space interpretation of these results, and \\
\noindent
(3) A homological analysis of the
cohomology of some of these groups for genus zero, one, and two
surfaces (possibly with marked points).
\vskip .2in

These notes address certain properties of configuration spaces of
surfaces, such as their connections to mapping class groups, as well as
their connections to classical homotopy theory that emerged over 15 years
ago. These constructions have various useful properties. For
example, they give easy ways to compute the cohomology of certain
related discrete groups, as well as give interpretations of this cohomology 
in terms of related mathematics.

\vskip .2in

In  particular certain explicit models for Eilenberg-MacLane
spaces of type $K(\pi,1)$ are given for certain kinds of 
braid groups and mapping class groups. For example,
we recall the classical result that 
configuration spaces of surfaces that are neither the 2-sphere nor
the real projective plane are $K(\pi,1)$ spaces. The
analogous configuration spaces for the the 2-sphere and
the real projective plane are not $K(\pi,1)$'s, but this is remedied
by considering natural actions of certain groups on these
surfaces and forming the associated Borel constructions.
\vskip .2in

For example, the group $SO(3)$ acts on the 2-sphere by rotations. Hence,
this group acts on the configuration space. 
Thus, there is an induced action of $S^3$, the non-trivial
double cover of $SO(3)$, on these configuration spaces. A common theme
throughout these notes is the structure of the Borel construction
for these types of actions. For example, the case of the  
$S^3$ Borel construction for the natural
action of $S^3$ on points in the 2-sphere or the projective plane
 give  $K(\pi,1)$ spaces whose fundamental group is the braid group of the
2-sphere or real projective plane respectively.  The resulting spaces are
elementary flavors of moduli spaces and may be described in
elementary terms as spaces of polynomials. These were investigated in
[].
\vskip .2in
One feature of the view here is that when calculations ``work'',
they do so easily, and give global descriptions of certain
cohomology groups. For example, the cohomology of the genus zero
mapping class group with marked points gives a (sometimes)
computable version of cyclic homology. Some concrete calculations
are given.

\vskip .2in
Also, in genus one, there is a version of a based
mapping class group. In this case the cohomology of these groups
with all possible marked points admits an accessible and simple
description. In additon, the genus 2 mapping class group
has a very simple ``configuration-like'' description from which the
torsion in the cohomology follows at once. Some of this has found
application
to the integral cohomology of $Sp(4,\mathbb{Z})$, the $4 \times 4$ 
integral matrices that
preserve the symplectic form of $\RR^4$.
Most of the work here is directed at calculating torsion
in the integral cohomology.
\vskip .2in

 TABLE OF CONTENTS
\begin{itemize}

\item Section   [1]:  Introduction
\item Section   [2]:  Definitions of configuration spaces and related
spaces
\item Section   [3]:  Braid groups and mapping class groups/ Smale's
theorem, Earle-Eells/ configuration spaces as homogeneous spaces
\item Section   [4]:  $K(\pi,1)$'s obtained from configuration spaces
\item Section   [5]:  Orbit configuration spaces ala'  Xicotencatl
\item Section   [6]:  Function spaces and configuration spaces/Dold-Thom
\item Section   [7]:  Labeled configuration spaces
\item Section   [8]:  Homological and topological splitings of function
spaces
\item Section   [9]:  Homological consequences
\item Section   [10]:  Homology of function spaces from a surface to the
2-sphere,
and the ``generalized Heisenberg group''
\item Section  [11]:  Based mapping class groups, particles on a punctured
torus, and automorphic forms
\item Section  [12]:  Homological calculations for braid groups
\item Section  [13]:  Central extensions
\item Section  [14]:  Geometric analogues of central extensions and
applications to
mapping class groups
\item Section  [15]:  The genus two mapping class group and the unitary
group
\item Section  [16]:  Lie algebras, and central extensions/
automorphisms of free groups and Artin's theorem, Falk-Randell theory
\item Section  [17]:  Samelson products, and Vassiliev invariants
of braids.
\end{itemize}

\section{Configuration spaces}
\label{sec: conf}

Throughout these notes, we will assume $X$ is a Hausdorff space
and furthermore if $X$ has a basepoint, we will assume it is
non-degenerate. (This means that the inclusion of the basepoint into
the space is a cofibration. We will discuss this more when we need it.) 
We will frequently also assume that $X$ is of the homotopy type of a
CW-complex.

	We will use $X^k$ to denote the Cartesian product of $k$ copies
of $X$. We will be interested in studying certain
``configuration spaces''
associated to $X$, so let us define these next.

\begin{defn}
\label{defn: configuration}
Given a topological space $X$, and a positive integer $k$, let 
$$
F(X,k) = \{ (x_1,\dots,x_k) \in X^k : x_i \neq x_j \text{ for } i \neq j \}.
$$
This is the $k$-configuration space of $X$.
\end{defn}

So we see that elements in the $k$-configuration space of $X$, correspond
to $k$ distinct, ordered points from $X$.

Now, it is easy to see that $\Sigma_k$, the symmetric group on $k$ letters,
acts on $X^k$ by permuting the coordinates. If we restrict this action
to $F(X,k)$, it is easy to check that it is free (no nonidentity element
fixes a point). Thus we can form the quotient space 
\begin{align}
\label{symconfig}
SF(X,k) = F(X,k)/\Sigma_k
\end{align}
and the quotient map $\pi : F(X,k) \rightarrow SF(X,k)$ is a covering map.
An element of $SF(X,k)$ is a set of $k$ distinct (unordered) points from $X$.

\begin{rem}
We have of course that $F(X,1)=SF(X,1)=X$.
(In these notes, $=$ will mean homeomorphic and we will use $\simeq$ to
stand for homotopy equivalent).
\end{rem}

\begin{rem}
It is easy to see that $F(-,k)$ defines a covariant
functor from the category of topological spaces and continuous
injective maps, to itself. This is not a homotopy functor. For example
the unit interval, $[0,1]$, is homotopy equivalent to a point. However
$F([0,1],2)$ is a nonempty space while $F(point,2)$ is an
empty set.
\end{rem}

\begin{defn}
Given a space $X$, we will find it convenient to let $Q_m$ denote a set
of $m$ distinct points in $X$.
\end{defn}

Notice that for $k \geq m \geq 1$ there is a natural map 
$$
\pikm : F(X,k) \rightarrow F(X,m)
$$
obtained by projecting to the first $m$ factors. 

This map is very useful in studying the nature of $F(X,k)$ especially
when $X$ is a manifold without boundary. This is due to the following
fundamental theorem:

\begin{thm}[Fadell and Neuwirth]
\label{thm: fib}
If $M$ is a manifold without boundary (not necessarily compact) and $k \geq
m \geq 1$, then the map $\pikm$ is a fibration with fiber $F(M - Q_m, k-m)$.
\end{thm}

We will use this theorem, to get our first insight into the nature of
these configuration spaces. 

\begin{defn}
A $K(\pi,1)$-space $X$ is a path connected space 
where $\pi_i(X)=0$ for $i \geq 2$ and 
$\pi_1(X)=\pi$. It is well known, that the homotopy type of such a space
is completely determined by its fundamental group (recall all our
spaces are of the homotopy type of a CW complex).
\end{defn}

Let us first look at configuration spaces
for 2-dimensional manifolds. 
Our first result, will be to show that most of these are
$K(\pi,1)$-spaces. 

\begin{thm}
\label{thm: Kpi1}
Let $M$ be either $\mathbb{R}^2$ or a closed 2-manifold of genus $\geq 1$
(not necessarily orientable).
$F(M - Q_m,k)$ has no higher homotopy for all $k \geq 1, m \geq 0$.
In other words $F(M - Q_m,k)$ is a $K(\pi,1)$-space.
\end{thm}
\begin{proof}
We will prove it by induction on $k$. First the case $k=1$. If $m=0$,
we just have to note 
that $M=F(M,1)$ is a $K(\pi,1)$-space and for 
$m \geq 1$, $M - Q_m \simeq \text{ bouquet of circles }$,
and so is a $K(F, 1)$-space where $F$ is a free group of finite rank.

So we can assume $k > 1$ and that the theorem holds for all smaller values. 
By theorem~\ref{thm: fib}, the map $\pi_{(k,1)} : F(M - Q_m, k)
\rightarrow F(M - Q_m, 1)$ is a fibration with fiber
$F(M - Q_{m+1}, k-1)$. However, by induction both the base
and the fiber of this fibration are $K(\pi,1)$-spaces, and so by the homotopy
long exact sequence of the fibration, we can conclude the total space
is also a $K(\pi,1)$-space and so we are done.

\end{proof}

It follows, from theorem~\ref{thm: Kpi1}, that for any closed 2-dimensional
manifold $M$ besides the sphere $S^2$ and the projective plane $\mathbb{R}P^2$,
the homotopy type of $F(M,k)$ is completely determined by its fundamental 
group. So our next goal should be to understand that. It turns out there
is quite a beautiful picture for the fundamental group of a configuration
space in terms of braids and we shall explore this in the next section. 

For now, let us state a lemma that can be used to ensure that a configuration
space is path connected so that one does not have to worry about base
points when talking about the fundamental group.

\begin{lem}
Let $M$ be a connected manifold (without boundary) such that $M$ remains
connected when punctured at $k-1 \geq 0$ points. Then $F(M,k)$ is
path connected. So in particular, if $M$ is a connected manifold
of dimension at least 2, all of its configuration spaces are path connected.
\end{lem}
\begin{proof}
The proof is by induction on $k$. When $k=1$, it follows easily from the
hypothesis. So we can assume $k>1$ and that we have proved it for smaller
values. Then by theorem~\ref{thm: fib}, 
$\pi_{k,1} : F(M,k) \rightarrow F(M,1)$ is a fibration with fiber
$F(M - Q_1, k-1)$. Thus by induction, both the base and the fiber are
path connected and hence so is the total space. 
\end{proof}

\section{Braid groups}

Let $I$ denote the unit interval $[0,1] \subset \mathbb{R}$.
If $X$ is a space, 
we can view $X \times \{0\}$ as the bottom of the ``cylinder'' $X \times I$,
and similarly $X \times \{1\}$ as the top.

Let $I_k$ be the space which consists of $k$ disjoint copies of $I$ where
the copies are labeled
from $1$ to $k$. Then 
let $0_i \in I_i$ be the point in the $i$th copy 
of $I$ corresponding to $0$ and similarly let $1_i \in I_i$ be the point in the
$i$th copy of $I$ corresponding to $1$.

Let $E=(e_1, \dots, e_k)$ be an element in $F(X,k)$, 
and let $\pi_I : X \times I
\rightarrow I$ be the projection map to the second factor. We are
now ready to define what we mean by a pure braid in $X$.

\begin{defn}
A pure $k$-stranded braid in $X$ (based at $E$) is a continuous, one to one
map $f: I_k \rightarrow X \times I$ which satisfies: \\
\noindent
(a) $\pi_I \circ f : I_k \rightarrow I$ is the identity map on each component
of $I_k$ and \\
\noindent
(b) $f(0_i)=(e_i,0)$, $f(1_i)=(e_i,1)$ for all $1 \leq i \leq k$.
\end{defn}

Let us look at an example to make the formal definition above intuitively clear.
Let us take $X=\mathbb{R}^2$ and $k=3$. We can picture $X \times I$ as
the subspace of $\mathbb{R}^3$ where
the $z$-coordinate satisfies $0 \leq z \leq 1$.
The following is a typical picture of a pure $3$-stranded braid of 
$\mathbb{R}^2$.

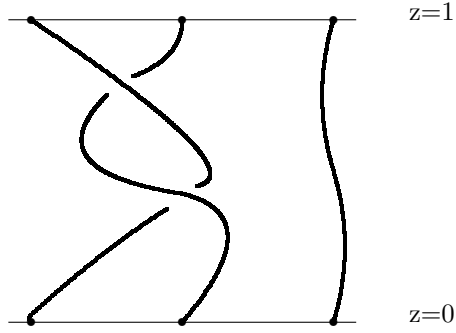
\setlength{\unitlength}{1.0 cm}
\begin{picture}(9,8)
\multiput(2,2)(2,0){3}{\circle*{0.1}}
\multiput(2,6)(2,0){3}{\circle*{0.1}}
\put(1.7,2){\line(1,0){4.6}}
\put(1.7,6){\line(1,0){4.6}}
\put(7,2){z=0}
\put(7,6){z=1}
\linethickness{0.4mm}
\qbezier(6,2)(6.3,3)(6,4)
\qbezier(6,4)(5.7,5)(6,6)
\qbezier(2,2.1)(3,3)(3.8,3.5)
\qbezier(2,2)(1.95,2.05)(2,2.1)
\qbezier(4.2,3.8)(5,4)(2,6)
\qbezier(4,2)(5.2,3.4)(4,3.7)
\qbezier(4,3.7)(2,4)(3,5)
\qbezier(3.35,5.25)(4,5.5)(4,6)
\end{picture}


It is obvious from the picture, that we would like to say two braids
are equivalent if we can ``deform'' one to the other. Thus we define:

\begin{defn}
Let $f_0$ and $f_1$ be two pure $k$-stranded braids based at $E$. Then
we say $f_0$ is equivalent to $f_1$ if there exists a homotopy 
$F: I_k \times I \rightarrow X \times I$ between them, such that $F$
restricted to $I_k \times \{t\}$ is a pure $k$-stranded braid for all $t \in I$. 
\end{defn}

It is easy to see that the equivalence above, indeed gives an equivalence
relation on the set of all pure $k$-stranded braids of $X$.

Now let us explore the natural correspondence between pure $k$-stranded
braids based at $E$ and loops in $F(X,k)$ based at $E$.

One can view such a loop as a map $\theta : I \rightarrow F(X,k)$ with 
$\theta(0)=\theta(1)=E$. However, $F(X,k) \subseteq X^k$ so we can take the 
components
of the map $\theta$ to get maps $\theta_i: I \rightarrow X$ for $1 \leq i
\leq k$. We can then define $f_i : I \rightarrow X \times I$ by
$$
f_i(t) = (\theta_i(t),t).
$$
Finally we can take these maps and put them together to get
a map $f: I_k \rightarrow X \times I$. 

It is a routine exercise to check
that the map $f$ obtained is indeed a pure $k$-stranded braid based at $E$
and that this establishes a one to one correspondence between loops
in $F(X,k)$ based at $E$ and pure $k$-stranded braids of $X$ based at $E$. 

Similarly, it is easy to check that two such loops are (base point preserving)
 homotopic if and only if the corresponding pure braids are equivalent.
Thus we see a one to one corespondence between $\pi_1(F(X,k) ; E)$ and
$PB_k(X ;E)$, the set of equivalence classes of pure $k$-stranded braids in $X$,
based at $E$. 

Of course this implies that $PB_k(X ; E)$ inherits the structure of a group 
from $\pi_1(F(X,k) ; E)$, but of course,
we can also describe this multiplication naturally on the level of the braids.  

Given two braids $f_0$ and $f_1$, we can  
think of $f_0$ as a braid between $X \times \{0\}$ and $X \times \{\frac{1}{2}
\}$ and
$f_1$ as a braid between $X \times \{\frac{1}{2}\}$ and $X \times \{1\}$, and then
$f_0 * f_1$ is the braid obtained by stacking the braid $f_1$ on top of the
braid $f_0$ as illustrated in the diagram below for the case 
$X=\mathbb{R}^2$ and $k=3$. It also follows from the earlier correspondence,
that the inverse of a pure braid
is just obtained by turning the braid upside down. 

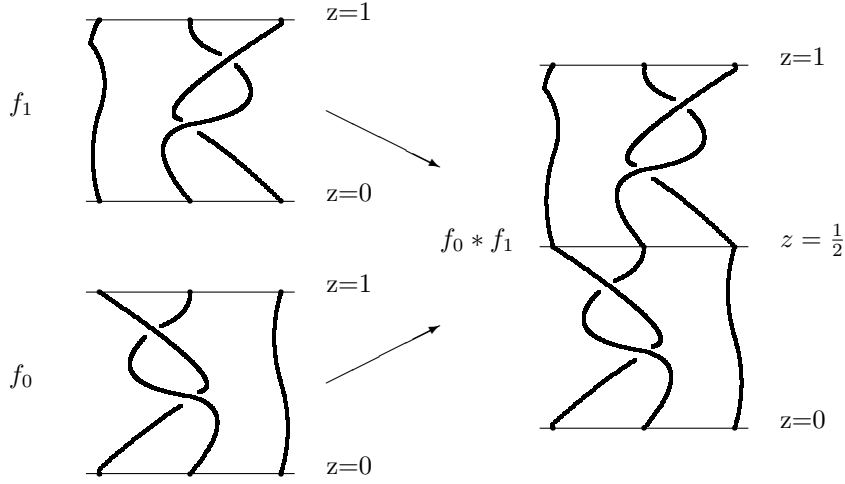
\setlength{\unitlength}{0.6cm}
\begin{picture}(18,13)
\put(0,4){$f_0$}
\multiput(2,2)(2,0){3}{\circle*{0.1}}
\multiput(2,6)(2,0){3}{\circle*{0.1}}
\multiput(2,8)(2,0){3}{\circle*{0.1}}
\multiput(2,12)(2,0){3}{\circle*{0.1}}
\multiput(12,3)(2,0){3}{\circle*{0.1}}
\multiput(12,7)(2,0){3}{\circle*{0.1}}
\multiput(12,11)(2,0){3}{\circle*{0.1}}
\put(9.5,7){$f_0 * f_1$}

\put(7,4){\vector(2,1){2.5}}
\put(7,10){\vector(2,-1){2.5}}

\put(11.7,11){\line(1,0){4.6}}
\put(17,11){z=1}
\put(1.7,8){\line(1,0){4.6}}
\put(1.7,12){\line(1,0){4.6}}
\put(7,8){z=0}
\put(7,12){z=1}
\put(11.7,3){\line(1,0){4.6}}
\put(11.7,7){\line(1,0){4.6}}
\put(1.7,2){\line(1,0){4.6}}
\put(1.7,6){\line(1,0){4.6}}
\put(7,2){z=0}
\put(7,6){z=1}
\linethickness{0.4mm}
\qbezier(6,2)(6.3,3)(6,4)
\qbezier(6,4)(5.7,5)(6,6)
\qbezier(2,2.1)(3,3)(3.8,3.5)
\qbezier(2,2)(1.95,2.05)(2,2.1)
\qbezier(4.2,3.8)(5,4)(2,6)
\qbezier(4,2)(5.2,3.4)(4,3.7)
\qbezier(4,3.7)(2,4)(3,5)
\qbezier(3.35,5.25)(4,5.5)(4,6)

\put(0,10){$f_1$}

\qbezier(2,8)(1.7,9)(2,10)
\qbezier(2,10)(2.3,10.8)(1.8,11.5)
\qbezier(1.8,11.5)(1.9,11.9)(2,12)
\qbezier(6,8)(5,9)(4.2,9.5)
\qbezier(3.8,9.8)(3,10)(6,11.9)
\qbezier(6,11.9)(6.05,11.95)(6,12)
\qbezier(4,8)(2.8,9.4)(4,9.7)
\qbezier(4,9.7)(6,10)(5,11)
\qbezier(4.65,11.25)(4,11.5)(4,11.9)
\qbezier(4,11.9)(4.05,11.95)(4,12)

\put(17,3){z=0}
\put(17,7){$z=\frac{1}{2}$}
\qbezier(16,3)(16.3,4)(16,5)
\qbezier(16,5)(15.7,6)(16,7)
\qbezier(12,3.1)(13,4)(13.8,4.5)
\qbezier(12,3)(11.95,3.05)(12,3.1)
\qbezier(14.2,4.8)(15,5)(12,7)
\qbezier(14,3)(15.2,4.4)(14,4.7)
\qbezier(14,4.7)(12,5)(13,6)
\qbezier(13.35,6.25)(14,6.5)(14,7)

\qbezier(12,7)(11.7,8)(12,9)
\qbezier(12,9)(12.3,9.8)(11.8,10.5)
\qbezier(11.8,10.5)(11.9,10.9)(12,11)
\qbezier(16,7)(15,8)(14.2,8.5)
\qbezier(13.8,8.8)(13,9)(16,10.9)
\qbezier(16,10.9)(16.05,10.95)(16,11)
\qbezier(14,7)(12.8,8.4)(14,8.7)
\qbezier(14,8.7)(16,9)(15,10)
\qbezier(14.65,10.25)(14,10.5)(14,10.9)
\qbezier(14,10.9)(14.05,10.95)(14,11)

\end{picture}

\begin{defn}
We will call $PB_k(X ; E)$, the pure $k$-stranded braid group of $X$.
We will usually surpress the basepoint $E$ when it is obvious and write
$PB_k(X)$.
\end{defn}

\begin{rem}
Notice of course that $PB_1(X)=\pi_1(X)$ for any path connected space $X$.
\end{rem}

So we see that in general $\pi_1(F(X,k))$
can be interpreted as the pure $k$-stranded braid group $PB_k(X)$, and
this will allow us to picture many relations among the elements of this
group.

We saw, in section~\ref{sec: conf} that there are many examples where
configuration spaces are $K(\pi, 1)$-spaces. This means that
in these cases, there will be a very strong connection between the 
configuration space
and the corresponding pure braid group. We will begin to exploit this
correspondence in the next section.

\begin{defn}
$\pi_1(SF(X,k))$ is called the $k$-stranded braid group of $X$ and
is denoted $B_k(X)$.
\end{defn}

Recall that we have a covering map $\pi : F(X,k) \rightarrow SF(X,k)$,
obtained by forming the quotient space under the free action of $\Sigma_k$
on $F(X,k)$. Assuming that $F(X,k)$ is path connected, from covering space 
theory, 
this implies that we have the
following short exact sequence of groups:

$$
1 \rightarrow PB_k(X) \rightarrow B_k(X) 
\overset{\lambda}{\rightarrow}
 \Sigma_k 
\rightarrow 1
$$

Fixing $E=(e_1,\dots,e_k) \in F(X,k)$, from covering space theory we can 
identify
elements of $\pi_1(SF(X,k) ; \pi(E))$ with homotopy classes of 
paths in $F(X,k)$ starting at
$E$ and ending at a point of $F(X,k)$ in the same $\Sigma_k$-orbit of $E$.
Using the same correspondence that associated loops in $F(X,k)$ with
pure braids, we see that these paths are naturally associated to $k$-stranded
braids of $X$ which start at the points $\{e_1, \dots, e_k \}$ at the bottom
of $X \times I$ and end at the same set of points on the top, but where
it is possible that the points get permuted. 

It is easy to check that there is a well defined multiplication on the 
equivalence classes of these braids, just as before, which corresponds
to the group structure of $\pi_1(SF(X,k))$. Thus $B_k(X)$ is just the
group of $k$-stranded braids in $X$ (based at some set of $k$ distinct points) 
which are allowed to induce a permutation of the points. The map $\lambda$
in the short exact sequence above is just the map that assigns to every
braid, the permutation it induces on the $k$ points. The pure braids
are hence exactly the elements in the kernel of this map. 

\begin{defn}
$B_k(\mathbb{R}^2)$ is called Artin's $k$-stranded braid group
and correspondingly, $PB_k(\mathbb{R}^2)$ is called Artin's $k$-stranded
pure braid group.
It follows from theorem~\ref{thm: Kpi1} that
$$
F(\mathbb{R}^2,k) = K(PB_k(\mathbb{R}^2),1)
$$
and
$$
SF(\mathbb{R}^2,k) = K(B_k(\mathbb{R}^2),1).
$$
\end{defn}

\section{Cohomology of groups}
\subsection{Basic concepts}

This section is provided to recall the relevant basic facts about the cohomology
of groups that we will need for the sequel. The reader is encouraged to read
when these results are needed and used.

Recall for every ``nice'' topological group G (Lie groups and all groups 
equipped with the
discrete topology are ``nice''), we have a universal principal $G$-bundle,
$$
EG \overset{\pi}{\rightarrow} BG
$$
where $EG$ is contractible and $G$ acts freely on it. 
The isomorphism classes of principal $G$-bundles over 
any 
paracompact space $X$, are then in bijective correspondence to $[X,BG]$,
the free homotopy classes of maps from $X$ into $BG$. Furthermore, it is
well known that $BG$ is unique up to homotopy equivalence.

Furthermore the correspondence $G \rightarrow BG$ can be made functorial,
so for every homomorphism of (topological) 
groups, $\lambda: G \rightarrow H$ we get a map
$B(\lambda): BG \rightarrow BH$. This map is always well defined up to 
free homotopy. 

If $G$ is discrete, then $\pi$ is a covering map and $BG$ is a $K(G,1)$-space.
Since $BG$ is unique up to homotopy equivalence, 
we may speak of the cohomology of a (discrete) group by defining
$$
H^*(G;M) = H^*((K(G,1);M).
$$
In general, we will also want to consider
``twisted'' coefficients in some $G$-module $M$.

One can also define the cohomology of a group, purely in the language of
homological algebra. 

\begin{defn}
Given a group $G$, and a $G$-module $M$, we define 
$$
H^*(G;M) = Ext^*_{\mathbb{Z}G}(\mathbb{Z}, M) 
$$
and
$$
H_*(G;M) = Tor_*^{\mathbb{Z}G}(\mathbb{Z}, M)
$$
where $\mathbb{Z}G$ is the integral group ring of $G$, and $\mathbb{Z}$
is given the structure of a $G$-module where every element of $G$ acts 
trivially.
\end{defn}

In other words, to calculate the cohomology of a group $G$, take any
 projective resolution of $\mathbb{Z}$ over the ring $\mathbb{Z}G$, 
apply $Hom_{\mathbb{Z}G}(-,M)$ to it, and calculate the cohomology of the resulting complex.

This definition, is equivalent to the topological one, see for 
example~\cite{Brown}.

\begin{rem}
It is an elementary fact that for any $G$-module $M$, one has
$$
H^0(G;M)=M^G=\{m \in M \, | \, gm=m \text{ for all } g \in G \}.
$$
$M^G$ is called the group of invariants of the $G$-module $M$.
\end{rem}

\begin{rem}
It is also easy to see that one has
$$
H_0(G;\mathbb{Z})=\mathbb{Z}$$
and
$$
H_1(G;\mathbb{Z})=G_{ab}
$$
where $G_{ab}$ is the abelianization of $G$.
\end{rem}

\subsection{Examples}
Here are some basic calculations:

\noindent
(a) Let $G=\mathbb{Z}/n$ be the cyclic group of order $n$. 
We can view $\mathbb{Z}/n$ as the $n$th
roots of unity in $\mathbb{C}$ and have it act on $\mathbb{C}^\infty$ by
coordinatewise multiplication. This action restricts to a free action on
 $S^\infty$, the
subspace of unit vectors whose coordinates are eventually zero and the
quotient space under this action is an infinite Lens space which is a
$K(\mathbb{Z}/n,1)$.
 
One calculates:
$$
H^k(\mathbb{Z}/n;\mathbb{Z}) = 
\begin{cases}
\mathbb{Z}/n &\text{ if } k > 0 , k \text{ even } \\
\mathbb{Z}   &\text{ if } k=0 \\
0 &\text{ otherwise }.
\end{cases}
$$ 

\noindent
(b) let $G=F_n$ be a free group of rank $n$.
In this case, a bouquet of $n$ circles is a $K(F_n,1)$. Since this is 
1-dimensional, it is easy
to see that $H^*(F_n;M)=0$ for all $* > 1$ and any coefficient $M$.
Furthermore, $H^1(F_n;\mathbb{Z})$ is a free abelian group of rank $n$.

\subsection{Cohomological dimension}
As we saw before, the free groups have the property
that $H^*(G;M)=0$ when $*>1$, so we might say they have cohomological dimension
$1$. In general we define:

\begin{defn}
The cohomological dimension of $G$ is denoted $cd(G)$. 
It is the maximum dimension $n$ such
that $H^*(G;M) \neq 0$ for some $G$-module $M$. Of course, if there is
no such maximum $n$, we say $cd(G)=\infty$. 
\end{defn}

Thus we see the cohomological dimension of a nontrivial free group is
one while $cd(\mathbb{Z}/n)=\infty$ for any 
$n \geq 2$.

Again, there is also a more direct description of cohomological dimension
using homological algebra.

\begin{defn}
Given a ring $R$, a projective resolution of a $R$-module $M$ is a 
exact sequence of $R$-modules
$$
\dots \rightarrow P_n \rightarrow \dots \rightarrow P_1 \rightarrow P_0
\rightarrow M \rightarrow 0
$$
where the $P_i$ are projective $R$-modules.
If $P_n \neq 0$ and $P_k=0$ for $k>n$ we say the resolution has length $n$.
\end{defn}

\begin{defn}
Given a ring $R$ and an $R$-module $M$, we define $proj_R(M)$, the projective
dimension of $M$, as the mimimum length of a projective resolution
of $M$. Of course, it is possible $proj_R(M)=\infty$.
\end{defn}

\begin{prop}
For any group $G$, 
$$
cd(G) = proj_{\mathbb{Z}G}(\mathbb{Z})
$$
\end{prop}
\begin{proof}
See \cite{Brown}.
\end{proof}

\begin{rem}
Given a $G$-module $M$, and a subgroup $H \le G$, we can regard
$M$ as a $H$-module in the obvious way. It is easy to check that
if $M$ is a free (projective) $\mathbb{Z}G$-module, then 
$M$ is also a free (projective) $\mathbb{Z}H$-module.
\end{rem}

\begin{prop}
For any group $G$, and subgroup $H \le G$ we have $cd(H) \leq cd(G)$.
Thus if $cd(G) < \infty$, $G$ is torsion-free.
\end{prop}
\begin{proof}
It is easy to check that a projective resolution of $\mathbb{Z}$ over
$\mathbb{Z}G$ restricts to a projective resolution of $\mathbb{Z}$ over
 $\mathbb{Z}H$. So the first statement
follows. If $G$ has nontrivial torsion it contains $\mathbb{Z}/n$ for some
$n \geq 2$ and the second statement follows from the first as 
$cd(\mathbb{Z}/n)=\infty$. 
\end{proof}

It is an elementary result that the only group that has cohomological 
dimension equal to zero 
is the trivial group. On the other hand we have seen that any nontrivial
free group has cohomological dimension 1. It is a deep result of Stallings
and Swan that the converse is also true, i.e.,
$$
cd(G) \leq 1 \iff G \text{ is a free group. }
$$
Thus we can think of $cd(G)$ as measuring how far a group $G$ is from being
free. If $cd(G) < \infty$, at least it is torsion-free, as we have seen.

We also have the following topological picture of cohomological dimension,
given by the following proposition, whose proof can be found in~\cite{Brown}.

\begin{prop}
Let the geometric dimension of $G$, denoted by $geod(G)$,
 be the minimum dimension of a $K(G,1)$-CW-complex. 
Then we have $geod(G)=cd(G)$ except possibly for the case
where $cd(G)=2$ and $geod(G)=3$.
\end{prop}

\begin{rem} Whether the exceptional case $cd(G)=2$ and $geod(G)=3$ can
occur is unknown. The conjecture that $cd(G)=geod(G)$ in general is known
as the Eilenberg-Ganea conjecture.
\end{rem}

We will also need the following basic theorem on subgroups of finite
index whose proof can be found in ~\cite{Brown}.

\begin{thm}
\label{thm: cdfinindex}
Let $G$ be a torsion-free group and $H$ a subgroup of finite index.
Then $cd(G)=cd(H)$.
\end{thm}

\subsection{$FP_\infty$ groups}

\begin{defn}
A group $G$ is of type $FP_n$ if there is a projective resolution
$\{ P_i \}_{i=0}^\infty$ of $\mathbb{Z}$ over $\mathbb{Z}G$ such that 
$P_j$ is finitely
generated for $j=0,\dots,n$. A group $G$ is of type $FP_\infty$ if it is
of type $FP_n$ for all $n$.
\end{defn}

It is elementary to show that every group is $FP_0$ and that a group
is of type $FP_1$ if and only if it is finitely generated.
So we can think of $FP_n$ for higher $n$ as strengthenings 
of the condition that $G$ be finitely generated. 

It is also true that if $G$ is finitely presented,
then $G$ is of type $FP_2$. For a long time, it was conjectured the converse
was true but Bestvina and Brady provided a counterexample - a group that
is $FP_2$ but not finitely presented.

\begin{rem}
It follows immediately, that if $G$ is of type $FP_n$ then for any
$G$-module $M$, which is finitely generated as an abelian group, we have
$H_i(G;M)$ and $H^i(G;M)$ finitely generated for all $i \leq n$.
\end{rem}

One has the following topological picture: 

\begin{prop}
A finitely presented group $G$ is of type $FP_n$ if and only if
there exists a $K(G,1)$-CW-complex such that the $n$-skeleton is finite. 
If $G$ is finitely presented, then it is of type $FP_\infty$ if and only if
there exists a $K(G,1)$-CW-complex such that the $n$-skeleton is finite
for all $n$.
\end{prop}

\begin{defn}
A group is of type $FP$ if it is of type $FP_\infty$ and has finite
cohomological dimension. This happens exactly when there is a projective
resolution $\{ P_i \}_{i=0}^\infty$ of $\mathbb{Z}$ over $\mathbb{Z}G$
which has finite length and such that each $P_i$ is finitely generated.
\end{defn}

If we are given such a finite projective resolution, it does not necessarily
mean that we can find a finite length resolution using free modules of finite
rank. Thus we define:

\begin{defn}
A group is of type $FL$ if there is a resolution of finite length for
$\mathbb{Z}$ over
$\mathbb{Z}G$ using finitely generated free modules.
\end{defn}

Again, we have a more concrete topological picture.

\begin{prop}
Let $G$ be a finitely presented group. Then $G$ is of type $FP$ if and only
if there exists a finitely dominated $K(G,1)$-CW-complex. Similarly,
$G$ is of type $FL$ if and only if there exists a finite $K(G,1)$-CW complex.
\end{prop} 

Thus for example $F_n$ is $FL$ as we can take $K(F_n,1)$
to be the bouquet of $n$ circles.

We will also find the following proposition useful:

\begin{prop}
\label{prop: seqFP}
If we have a short exact sequence of groups
$$
1 \rightarrow \Gamma_0 \rightarrow \Gamma \rightarrow \Gamma_1 \rightarrow 1
$$
then if $\Gamma_0$ and $\Gamma_1$ are of type $FP$ (respectively $FL$) then
so is $\Gamma$.
\end{prop}

\begin{defn}
If a group $G$ is of type $FP$, we define $\chi(G)$, the Euler characteristic 
of $G$ to be the Euler characteristic of a $K(G,1)$. This makes sense
since $H^*(G;\mathbb{Z})$ is finitely generated in each dimension, and
is zero for $* > cd(G)$.
\end{defn}

The following proposition recalls well-known facts about subgroups of finite
index.

\begin{prop}
\label{prop: Eulerfinindex}
If $G$ is a torsion-free group, and $H$ is a subgroup of finite 
index. Then 
$G$ is of type $FP$ if and only if $H$ is. Furthermore in this case, 
$$
\chi(H) = |G:H|\chi(G).
$$ 
\end{prop}

\subsection{The five lemma}
We state a refined form of the five lemma here as a 
convenient quick reference. The reader can supply the usual
diagram chasing proof if they feel inclined!

\begin{lem}
\label{lem: 5lem}
Given a commutative diagram of (not neccessarily abelian) groups
and homomorphisms:

$$
\begin{CD}
H_1 @>>> H_2 @>>> H_3 @>>> H_4 @>>> H_5 \\
@VV \mu_1 V @VV \mu_2 V @VV \mu_3 V @VV \mu_4 V @VV \mu_5 V \\
G_1 @>>> G_2 @>>> G_3 @>>> G_4 @>>> G_5
\end{CD}
$$
where the rows are exact, the following is true: \\

\noindent
(a) If $\mu_2,\mu_4$ are monomorphisms and $\mu_1$ is an
epimorphism then $\mu_3$ is a monomorphism. \\

\noindent
(b) If $\mu_2,\mu_4$ are epimorphisms and $\mu_5$ is a
monomorphism then $\mu_3$ is an epimorphism. \\

\noindent
(c) If $\mu_1,\mu_2,\mu_4$ and $\mu_5$ are isomorphisms,
then so is $\mu_3$.
\end{lem}

\subsection{The Lyndon-Hochschild-Serre (LHS) spectral sequence}
For any short exact sequence of groups
$$
1 \rightarrow N \overset{i}{\rightarrow} G \overset{\pi}{\rightarrow}
 Q \rightarrow 1
$$
and $G$-module $M$, we have a $E_2$-spectral sequence abutting to 
$H^*(G;M)$ and whose $E_2$ term is given by
$$
E_2^{p,q} = H^p(Q;H^q(N;M)).
$$
Of course, there is also a similar spectral sequence in homology.
(Here, recall that in general, $H^q(N;M)$ is a nontrivial
$Q$-module where the $Q$ action is induced by the conjugation
action of $G$ on $N$.)

This spectral sequence has a topological origin, for if you have
such an exact sequence of groups, then 
$$
BN \overset{Bi}{\rightarrow} BG \overset{B\pi}{\rightarrow} BQ
$$
is a fibration, and the LHS-spectral sequence is nothing more than the
Serre spectral sequence for this fibration, possibly using twisted coefficients.

\section{Polyfree groups}

\begin{defn}
\label{defn: normalseries}
A normal series for a group $G$ is a sequence of subgroups 
$$
1=G_0 \le G_1 \le \dots \le G_n=G
$$
where each $G_i$ is normal in $G$. $\Gamma_i=G_i/G_{i-1}$ 
is referred to as the $i$th factor for $1 \leq i \leq n$. 
The length of a normal series is 
the number of nontrivial factors.
\end{defn}

\begin{defn}
A polyfree group is a group which has a normal series where all the factors
are finitely generated free groups. Such a normal series will be refered to
as a polyfree series and the rank of the $k$th factor will be called $d_k$,
the $k$th exponent.
\end{defn}

\begin{rem}
Apriori, we do not know whether every polyfree series of a given polyfree group
will have the same length, nor do we know if the exponents may vary with
the polyfree series chosen.
We will return to this point shortly.
\end{rem}

Our interest in polyfree groups stems from the fact that many of the pure
braid groups we have considered up to now, are polyfree.

\begin{thm}
\label{thm: motpoly}
Let $M$ be equal to $\mathbb{R}^2$ or a closed 2-manifold of genus $g \geq 1$. 
Then $PB_k(M-Q_m)$ is polyfree for any $m,k \geq 1$. (Notice that
we do have to puncture the manifold at least once in general). Furthermore
there is a polyfree series of length $k$ where the rank of the $i$th factor
is equal to 
$2g+m-1 -i + k$ for $1 \leq i \leq k$. (These formulas also work for 
$M=\mathbb{R}^2$ if we set $g=\frac{1}{2}$).
\end{thm}
\begin{proof}
As usual, we will induct on $k$. For $k=1$, $F(M-Q_m,k)=M-Q_m$ is a
$K(F_{2g+m-1},1)$ where $F_{2g + m - 1}$ is a free group of rank $2g+m-1$. 
Thus the theorem
follows easily in this case. So we may assume $k>1$ and that the theorem
is proven for smaller values of $k$. Then by theorem~\ref{thm: fib}, we
have
$$
\pi_{k,k-1} : F(M-Q_m,k) \rightarrow F(M-Q_m,k-1)
$$
is a fibration with fiber $F(M-Q_{m+k-1},1)$. Since these spaces have
trivial $\pi_2$ by theorem~\ref{thm: Kpi1}, we get a short exact sequence
of groups
$$
1 \rightarrow F_{2g + m - 2 + k} \rightarrow PB_k(M-Q_m) \rightarrow
PB_{k-1}(M-Q_m) \rightarrow 1.
$$
By induction $PB_{k-1}(M-Q_m)$ has a polyfree series of length $k-1$
with the claimed exponents and so it follows from the short exact sequence of 
groups above, that
$PB_k(M-Q_m)$ has a polyfree series of length $k$. It is an easy exercise
which will be left to the reader, to check that the factors have the ranks 
claimed.
\end{proof}

\begin{rem}
For $k \geq 2$, by considering the map 
$\pi_{k,1} : F(\mathbb{R}^2,k) \rightarrow 
F(\mathbb{R}^2,1)=\mathbb{R}^2$ we see that 
$$
F(\mathbb{R}^2,k) = \mathbb{R}^2 \times F(\mathbb{R}^2-Q_1,k-1)
$$
so it follows that $PB_k(\mathbb{R}^2)=PB_{k-1}(\mathbb{R}^2-Q_1)$
and so the pure $k$-stranded Artin braid group is polyfree with a polyfree 
series
of length $k-1$ given by the theorem above.
\end{rem}

So now we have a lot of motivation to study the class of polyfree groups.
Let us begin with some elementary results.
First recall the following classical theorem of P. Hall:

\begin{thm}[P. Hall]
\label{thm: Hall}
If 
$$
1 \rightarrow N \rightarrow G \rightarrow Q \rightarrow 1 
$$
is a short exact sequence of groups, with $N$ and $Q$ finitely presented,
then $G$ is finitely presented.
\end{thm}
\begin{proof}
See e.g.~\cite{Robin}
\end{proof}

\begin{prop}
\label{prop: baspoly}
Suppose $G$ is a polyfree group with a polyfree series of length $n$.
Then: \\
(a) $cd(G) \leq n$ and in particular $G$ is torsion-free. \\
(b) $G$ is of type $FL$. \\
(c) $G$ is finitely presented. \\
(d) If the exponents for the polyfree series satisfy $d_k \geq 2$ for
all $1 \leq k \leq n$ then $G$ has trivial center.
\end{prop} 
\begin{proof}
We will induct on $n$, the length of the polyfree series. If $n=1$, then
$G$ is free of rank $d_1$. The theorem then follows readily once we note
that a free group of rank $d_1 \geq 2$ has trivial center. (Since for
example it is a nontrivial free product.)

So we can assume $n>1$ and that we have proved the proposition for all smaller
$n$. Let $1=G_0 \le G_1 \le \dots \le G_n=G$ be the given polyfree series.
Notice we have a short exact sequence
$$
1 \rightarrow G_{n-1} \rightarrow G \overset{\pi}{\rightarrow} 
F_{d_n} \rightarrow 1
$$
and that $G_{n-1}$ is polyfree with a series of length $n-1$.
Thus by induction $G_{n-1}$ is $FL$ and finitely presented and hence
so is $G$ by theorem~\ref{thm: Hall} and proposition~\ref{prop: seqFP}.
So we have proven (b) and (c). 

For (d), if we assume $d_k \geq 2$ for all
$k$ then by induction $G_{n-1}$ has trivial center and of course so
does $F_{d_n}$. Now if $c \in G$ is a central element, then it is easy
to see that $\pi(c)$ will be central in $F_{d_n}$ and hence trivial.
Thus $c \in G_{n-1}$ and so $c$ is a central element in $G_{n-1}$ which
finally lets us conclude $c=1$. Thus we have proven (d).

So it remains only to prove (a). By induction, $cd(G_{n-1}) \leq n-1$.
Now suppose $M$ is any $G$-module. Then we can apply the LHS-spectral
sequence to the short exact sequence above. 
Now $E_2^{p,q} = H^p(F_{d_n};H^q(G_{n-1};M))$ is zero if either $p > 1$
or $q >n-1$, as $cd(G_{n-1}) \leq n-1$. Thus we see there can be
no nontrivial differentials in this spectral sequence and so $E_2=E_\infty$.

However this spectral sequence converges to $H^*(G;M)$ and so we can conclude
easily that $H^*(G;M)=0$ for $* > n$. Since this holds for any $G$-module
$M$, we can conclude $cd(G) \leq n$.
\end{proof}

\begin{rem}
Given a polyfree group $G$ with a polyfree series of length $n$, 
proposition~\ref{prop: baspoly} guarantees the existance of a resolution of
$\mathbb{Z}$ over $\mathbb{Z}G$ of length less than or equal to $n$,
using finitely generated, free $\mathbb{Z}G$-modules.

A nice resolution of this sort 
was constructed in ~\cite{Dan} using Fox free derivatives
in the case where the polyfree group satisfies an extra splitting condition
which we will consider later. The naturality of this resolution was also
exploited to recover many interesting representations which we will
also look at later.
\end{rem}

Next we consider subgroups of finite index in a polyfree group.

\begin{prop}
\label{prop: polyfin}
Let $G$ be a polyfree group with a polyfree series of length $n$ and
let $H$ be a subgroup of finite index in $G$. Then $H$ is polyfree
with a polyfree series of length $n$.
\end{prop}
\begin{proof}
We proceed by induction on $n$. If $n=1$ then $G$ is a nontrivial free group
of finite rank. It follows then from the Nielsen-Schreier theorem 
(see \cite{Robin}) that
$H$ is also free of finite rank and so this case follows.

So we may assume $n > 1$ and that we have proved the proposition for smaller
values. Let $1 = G_0 < G_1 < \dots < G_n =G$ be the polyfree series for $G$.
Then if we set $H_i=G_i \cap H$ for $0 \leq i \leq n$, it is easy to check
$1=H_0 < H_1 < \dots <H_n=H$ is a normal series for $H$. Furthermore,
if $H_i/H_{i-1} \rightarrow G_i/G_{i-1}$ is the map induced by inclusion
of $H$ in $G$, it is easy to check this map is a well defined monomorphism
(injection) of groups. On the other hand the image of this map has finite index
in $G_i/G_{i-1}$ as $G_i/H_i$ injects (as sets) into the finite set $G/H$ 
via the map induced by inclusion. Thus we see each $H_i/H_{i-1}$ can be
viewed as a subgroup of finite index in the nontrivial finitely generated 
free group
$G_i/G_{i-1}$ and hence is itself a nontrivial free group of finite rank,
again by the Nielsen-Schreier theorem. 
Thus $1=H_0 < H_1 < \dots <H_n=H$ is a polyfree series for
$H$ of length $n$.
\end{proof}

\begin{rem}
An arbitrary subgroup of a free group of finite rank need not have finite rank
and so we see that arbitrary subgroups of a polyfree group need not be
polyfree.
\end{rem}

Now let us show that for any given polyfree group $G$, all polyfree series
for $G$ have the same length. First we will need a technical lemma.
Let $\mathbb{F}_2$ be the field with two elements.

\begin{lem}
\label{lem: lowbound}
Let $G$ be a polyfree group with a polyfree series 
$1=G_0 < G_1 <\dots<G_n=G$ of length $n$.
Then there exists a subgroup $T$ of finite index such that 
$H^n(T;\mathbb{F}_2) \neq 0$. Furthermore, $T$ contains $G_1$ and
$\chi(T) = \chi(G_1)\chi(T/G_1)$. Thus $\chi(G)=\chi(G_1)\chi(G/G_1)$.
\end{lem} 
\begin{proof}
As usual we will proceed by induction on $n$. If $n=1$, $G$ is a nontrivial 
free group of finite
rank. Thus we can take $T=G=G_1$ as $H^1(G;\mathbb{F}_2)=Hom(G;\mathbb{F}_2)$
is nonzero. (The statement about Euler characteristics follows as
$\chi(1)=1$.) 

So we can assume $n>1$ and that we have proved the lemma for smaller $n$.
Then notice that $G/G_1$
has a polyfree series of length $n-1$.

Now $G$ acts by conjugation on $G_1$ and hence on the finite vector space
$$
H^1(G_1;\mathbb{F}_2)=Hom(G_1,\mathbb{F}_2).
$$ 
If we let $K$ be the kernel of this action on
$H^1(G_1;\mathbb{F}_2)$, then $K$ is normal and of finite index in $G$
and furthermore $G_1 \subseteq K$.

Now $K/G_1 \subseteq G/G_1$ is a subgroup of finite index and so by 
proposition~\ref{prop: polyfin}, $K/G_1$ is polyfree with
a polyfree series of length $n-1$. So, by hypothesis, we may find $T/G_1$ with 
$H^{n-1}(T/G_1;\mathbb{F}_2) \neq 0$ and $T$ a subgroup of $K$ of finite index.
However $T/G_1 \subseteq G/G_1$ and so $cd(T/G_1) \leq n-1$ by 
proposition~\ref{prop: baspoly}. So if we look at the LHS-spectral sequence
for the short exact sequence
$$
1 \rightarrow G_1 \rightarrow T \rightarrow T/G_1 \rightarrow 1
$$
we see that $E_2^{p,q}=H^p(T/G_1;H^q(G_1;\mathbb{F}_2))$ is zero if
$p>n-1$ or $q>1$ and $E_2^{n-1,1}$ is nonzero. (Here we use that 
$T \subseteq K$ so that 
$E_2^{*,*}=H^*(T/G_1;\mathbb{F}_2) \otimes H^*(G_1;\mathbb{F}_2)$.)
It is easy to see that
we must have $E_2^{n-1,1}=E_{\infty}^{n-1,1}$ and so we can conclude
$H^n(T;\mathbb{F}_2)$ is nonzero.
$T$ is obviously of finite index in $G$ so it only remains to prove the
statement about $\chi(T)$.

In the spectral sequence above, the $E_2$ term has Euler characteristic
$$
\chi(G_1)\chi(T/G_1).
$$
On the other hand the $E_{\infty}$ term abuts
to $H^*(T;\mathbb{F}_2)$ which has Euler characteristic
$\chi(T)$ and so the equality $\chi(T)=\chi(G_1)\chi(T/G_1)$ follows
from the general fact that the Euler characteristic remains constant
on each page of the LHS-spectral sequence. Thus our induction on $n$ goes
through. 

Finally the equality $\chi(G)=\chi(G_1)\chi(G/G_1)$ follows from
the previous one which had $T$ instead of $G$,
 theorem~\ref{prop: Eulerfinindex}, and the fact that
the index of $T$ in $G$ is the same as the index of $T/G_1$ in $G/G_1$.

\end{proof}

\begin{prop}
Let $G$ be a polyfree group with a polyfree series of length $n$
and exponents $d_k$, $1 \leq k \leq n$.
Then $cd(G)=n$ and so every polyfree series of $G$ has length $n$.
Furthermore $\chi(G)=\prod_{k=1}^n(1-d_k)$.
\end{prop}
\begin{proof}
By proposition~\ref{prop: baspoly}, we have $cd(G) \leq n$.
On the other hand, by lemma~\ref{lem: lowbound}, we have a subgroup $T$ of
$G$ of finite index, such that $cd(T) \geq n$. Thus the first
part follows as $cd(G)=cd(T)$ by proposition~\ref{thm: cdfinindex}.

Now let us prove the formula for $\chi(G)$ by induction on $n$. For $n=1$,
it follows as a bouquet of $d_1$ circles is a $K(G,1)$. So as usual we
can assume $n>1$ and that we have the formula for smaller $n$.

 Again by lemma~\ref{lem: lowbound}, we have
$\chi(G) = \chi(G_1)\chi(G/G_1) = (1-d_1)\chi(G/G_1)$.
The formula now follows readily once we observe that $G/G_1$ has a polyfree
series of length $n-1$ with exponents $d_2,\dots,d_n$ and so by induction,
$$
\chi(G/G_1)=\prod_{k=2}^n(1-d_k).
$$
\end{proof}
 
\begin{rem}
So from what we have seen, all polyfree series for a given polyfree
group $G$, have the same length which is equal to $cd(G)$. However
the exponents of different polyfree series need not be the same.

For example we saw that $PB_4(\mathbb{R}^2)=PB_3(\mathbb{R}^2 - Q_1)$
is polyfree with exponents $(3,2,1)$ by theorem~\ref{thm: motpoly}.
However it is known that this group also admits a polyfree series with
exponents $(5,2,1)$ (see \cite{Dan}).

However $\prod_{k=1}^n(1-d_k)$ is independent of the polyfree series chosen
as it is equal to $\chi(G)$ as we have seen. 
\end{rem}

\section{Configuration spaces for the 2-sphere}
	We have seen in theorem ~\ref{thm: Kpi1} 
that for any closed 2-dimensional manifold $M$ besides
the sphere $S^2$ and the projective plane $\mathbb{R}P^2$, $F(M,k)$
is a $K(PB_k(M),1)$. Thus the configuration space is a good model
for the pure braid group of $M$.
	Now let us see what we can say in the case where $M$ is $S^2$ or
$\mathbb{R}P^2$.  
 	Let us first look at $S^2$. First note that $SO(3)$ acts naturally
on $\mathbb{R}^3$ and this action preserves the unit sphere $S^2$. 
Hence $SO(3)$ acts naturally on $F(S^2,k)$ for any $k$ via a diagonal action.
We will describe these configuration spaces now up to 
homotopy equivalence (which will be denoted by $\simeq$).
In fact we will find homotopy equivalences which preserve the $SO(3)$ actions.

\begin{prop}
\label{pro: configsphere}
$$
F(S^2,k) \simeq
\begin{cases}
S^2 &\text{ if } k=1,2 \\
SO(3) &\text{ if } k=3 \\
SO(3) \times F(S^2 - Q_3, k-3) &\text{ if } k>3.
\end{cases}
$$
Furthermore, there are homotopy equivalences which are equivariant with
respect to the natural $SO(3)$ actions on $F(S^2,k)$ and the action of
$SO(3)$ on itself by left multiplication. (Here $SO(3)$
acts on $SO(3) \times F(S^2 - Q_3, k-3)$ by acting entirely on the left
factor via left multiplication.) 
\end{prop}
\begin{proof}
The $k=1$ case is trivial so let us look at the $k=2$ case first. 
By theorem~\ref{thm: fib}, 
$\pi: F(S^2,2) \rightarrow F(S^2,1)=S^2$ is a fibration with
contractible fiber $F(S^2 - Q_1, 1)=R^2$. Thus $\pi$ is a homotopy
equivalence and it is trivial to see it is equivariant with respect to the
natural actions of $SO(3)$ on $F(S^2,2)$ and $S^2$.

Now let us look at the case where $k=3$.
Again by theorem~\ref{thm: fib}, there is a fibration
$$
\pi: F(S^2,3) \rightarrow S^2
$$ 
obtained by projecting onto the first factor.

Fix a point $p_1 \in S^2$. Let $F=\pi^{-1}(p_1)$, the fiber over the point
$p_1$. 
Then again by theorem~\ref{thm: fib},
since $F=F(S^2 - Q_1,2)$ there is a fibration 
$$
\hat{\pi} : F \rightarrow (S^2-Q_1)=
\mathbb{R}^2
$$ 
with fiber $S^2-Q_2$, obtained by projecting onto a factor. 

Since
the base space of the fibration $\hat{\pi}$ is contractible, we conclude that
$$
F = \mathbb{R}^2 \times (S^2 - Q_2) \simeq S^1.
$$

Now we will define a map $\lambda : SO(3) \rightarrow F(S^2,3)$. Fix
$\bar{p}=(p_1,p_2,p_3) \in F(S^2,3)$. Specifically, we will take
$p_1=(0,0,1), p_2=-p_1$ and $p_3=(1,0,0)$. Then we can set
$$
\lambda(\alpha) = (\alpha(p_1),\alpha(p_2),\alpha(p_3))
$$ 
for all $\alpha \in
SO(3)$. 

It is easy to see that $\lambda$ is a $SO(3)$-equivariant, continuous
map from $SO(3)$ to $F(S^2,3)$. We now want to show it is a homotopy
equivalence. Let $I$ denote the isotropy group of the point $p_1 \in S^2$
under the $SO(3)$ action. Then $I$ is isomorphic to $SO(2)$ as a topological
group and $\lambda|_I$ maps $I$ into $F$. 

\begin{claim}
The map $\lambda|_I : I \rightarrow F$ is a homotopy equivalence.
\end{claim}
\begin{proof}
The action of $I$ on $S^2$ is given
by rotation about the $z$-axis fixing $p_1$ which can be thought of as
the north pole of $S^2$. Notice that since we chose
$p_2=-p_1$, any element of $I$ will also fix $p_2$ (which is the south pole).
Recall we had a fibration $\hat{\pi}: F \rightarrow S^2 - \{p_1\}$ and let
us set $F' = \hat{\pi}^{-1}(p_2)$, the fiber above the point $p_2$. We saw
above that $F' = S^2 - \{p_1,p_2\} \simeq S^1$ 
and that the 
inclusion
of $F'$ into $F$ is a homotopy equivalence. 

Now notice, that $\lambda|_I$ actually maps $I$ into $F'$ since any element
of $I$ fixes $p_2$. If we let $\alpha_{\theta}$ denote the element of
$I$ which corresponds to rotation through an angle $\theta$, then the
loop $\{ \alpha_{\theta} : 0 \leq \theta \leq 2\pi \}$ gets mapped
by $\lambda$ to the generator of $\pi_1(F')$. Thus we see
that $\lambda |_{I} : I \rightarrow F'$ is a homotopy equivalence since
both $I$ and $F'$ are homotopic to $S^1$ and $\lambda$ is surjective
on the $\pi_1$ level. 
Hence $\lambda |_{I} : I \rightarrow F$ is a homotopy equivalence
since the inclusion of $F'$ into $F$ is a homotopy equivalence.  
\end{proof}

Now one can check easily that we get the following commutative diagram:

$$
\begin{CD}
I @>>> SO(3) @> f >> S^2 \\
@VV \lambda |_I V @VV \lambda V @VV Id V \\
F @>>> F(S^2,3) @> \pi >> S^2
\end{CD}
$$
where $Id$ is the identity map and $f(\alpha) = \alpha(p_1)$ for $\alpha \in
SO(3)$. Furthermore, each row is a fibration. Comparing  
the long exact sequence in homotopy for the two fibrations, one sees
that the vertical maps induce an isomorphism on homotopy groups
on the base and fiber levels of these fibrations and hence (by the five lemma)
 also on the level
of total spaces. Thus, 
$\lambda: SO(3) \rightarrow F(S^2,3)$
induces an isomorphism on homotopy groups and hence is a homotopy
equivalence and so we have the result for the case $k=3$.
 
	Now for the final case when $k > 3$. Let $\bar{p}=(p_1,\dots,p_k)
\in F(S^2,k)$ extend $(p_1,p_2,p_3)$ chosen in the previous case.
Then again we can define a map 
$$
\lambda: SO(3) \times F(S^2 - \{p_1,p_2,p_3\},k-3) \rightarrow F(S^2,k)
$$
by 
$$
\lambda(\alpha, (a_4,\dots, a_k)) = (\alpha(p_1),\alpha(p_2),\alpha(p_3),
\alpha(a_4),\dots,\alpha(a_k))
$$
for $\alpha \in SO(3)$ and $(a_4,\dots,a_k) \in F(S^2-\{p_1,p_2,p_3\},k-3)$.

It is easy to see that this map is continuous and $SO(3)$-equivariant
under the $SO(3)$-actions described in the statement of this lemma.
Furthermore, we get the following commutative diagram:

$$
\begin{CD}
\pi^{-1}(Id) @>>> SO(3) \times F(S^2-\{p_1,p_2,p_3\},k-3)
@> \pi >> SO(3) \\
@V Id VV @V \lambda VV @V \lambda VV \\
\pi_{k,3}^{-1}(p_1,p_2,p_3) @>>> F(S^2,k) @> \pi_{k,3} >> F(S^2,3)
\end{CD}
$$
where as usual $Id$ stands for an identity map and $\pi$ is projection onto
the first factor. Furthermore both rows are fibrations. We have
seen that the vertical map on the base level is a homotopy equivalence
and also trivially the vertical map on the fiber level is also a homotopy
equivalence and so it follows that $\lambda$ is indeed a homotopy equivalence
(by looking at the long exact sequences in homotopy for the fibrations,
and using the five-lemma).
So this concludes the proof of the lemma.
\end{proof}

Proposition~\ref{pro: configsphere} has the following immediate
corollary:

\begin{cor}
$$
PB_k(S^2) =
\begin{cases}
1 &\text{ if } k=1,2 \\
\mathbb{Z}/2\mathbb{Z} &\text{ if } k=3 \\
\mathbb{Z}/2\mathbb{Z} \times PB_{k-3}(\mathbb{R}^2-Q_2)
&\text{ if } k>3
\end{cases}
$$
\end{cor}

Notice that $F(S^2,k)$ is never a $K(\pi,1)$-space due to the higher
homotopy in $SO(3)$ and $S^2$. However, we will see in the next subsection,
that one can perform a natural construction to $F(S^2,k)$ for $k \geq 3$
and obtain a useful $K(\pi,1)$-space. 

\subsection{The Borel Construction}
We will now recall a very important construction. First some
elementary definitions:

\begin{defn}
If $G$ is a topological group with identity element $e$, then a (left) 
$G$-space is a space $X$
together with a continuous map $\mu : G \times X \rightarrow X$ which 
satisfies: \\
\noindent
(a) $g_1 \cdot (g_2 \cdot x) = (g_1g_2)\cdot x$ and \\
\noindent 
(b) $e \cdot x = x$ \\ 
\noindent
for all $g_1,g_2 \in G$, $x \in X$. (Here we use the notation
$g \cdot x$ for $\mu(g,x)$ and $e$ denotes the identity element of $G$). 
Of course, similarly there is a notion
of right $G$-space where the action is on the right instead of on
the left.
\end{defn}

\begin{defn}
Keeping the notation above, we say a $G$-space $X$ is free if 
$g \cdot x = x$ implies that $g=e$.
\end{defn}

Recall that for a
``nice'' topological group $G$, there is a universal $G$-bundle:

$$
G \rightarrow EG \rightarrow BG
$$

where $EG$ is a contractible free (right) $G$-space.

\begin{defn}
Given a (left) $G$-space $X$, one can perform the 
Borel construction:
$$
EG \times_G X = EG \times X / \sim
$$
where $\sim$ is the equivalence relation generated by
$ (tg,x) \sim (t,gx) $ for all $t \in EG, x \in X, g \in G$. 
\end{defn}

\begin{rem}
\label{rem: basborel}
It is easy to see from this definition, that
$EG \times_G G = EG \simeq *$.
Here $G$ is viewed as a left $G$-space via left multiplication on
itself and we are using the standard notation of $*$ to denote
a space which consists of a single point. 
\end{rem}

The next lemma collects some well-known elementary facts about
the Borel construction:

\begin{lem}
Let $G$ be a topological group and $X$ be a $G$-space
then 
$$
\pi_1: EG \times_G X \rightarrow BG=EG/G,
$$ 
the map induced
by projection onto the first factor, is a bundle
map with fiber $X$.

Furthermore, if $X$ is a free $G$-space, then the map
$$
\pi_2: EG \times_G X \rightarrow X/G,
$$ 
induced by projection
onto the second factor is a bundle map with fiber $EG$. Since
$EG$ is contractible, this implies $\pi_2$ is a (weak) homotopy equivalence.
\end{lem}

Let $G_1,G_2$ be two topological groups and let 
$\mu : G_1 \rightarrow G_2$ be a homomorphism of topological groups.
Suppose further that you are given a $\mu$-equivariant map 
$$
f: X_1 \rightarrow X_2,
$$
where $X_i$ is a left $G_i$-space for $i=1,2$. In other words, $f$
satisfies:
$$
f(g \cdot x) = \mu(g) \cdot f(x)
$$ 
for all $g \in G_1, x \in X_1$.

Furthermore recall, from the functoriality of the construction of the
universal principal $G$-bundle, one has the map 
$E(\mu): EG_1 \rightarrow EG_2$
which is also $\mu$-equivariant which induces the map 
$B(\mu): BG_1 \rightarrow BG_2$ on the level of the quotient spaces.

Thus one can look at
$$
E(\mu) \times f : EG_1 \times X_1 \rightarrow EG_2 \times X_2.
$$
It is easy to check that this map respects the equivalence relations
used in forming the 
respective Borel constructions and thus one gets a map:

$$
E(\mu) \bar{\times} f : EG_1 \times_{G_1} X_1 \rightarrow EG_2 \times_{G_2}
X_2.
$$ 

There is a commutative diagram in this context that is very useful.
We state it as the next lemma.

\begin{lem}
\label{lem: CD for Borel}
Given $\mu: G_1 \rightarrow G_2$ a homomorphism of topological groups
and $f: X_1 \rightarrow X_2$ a $\mu$-equivariant map one has
the following commutative diagram:
$$
\begin{CD}
X_1 @>>> EG_1 \times_{G_1} X_1 @>>> BG_1 \\
@VV f V @VV E(\mu) \bar{\times} f V @VV B(\mu) V \\
X_2 @>>> EG_2 \times_{G_2} X_2 @>>> BG_2 
\end{CD}
$$
where each row is a fiber bundle.
So in particular, if $B(\mu)$ and $f$ are weak homotopy equivalences,
then so is $E(\mu) \bar{\times} f$. 
\end{lem}
\begin{proof}
It is a routine exercise to check that the diagram commutes.
The last sentence then follows once again from the five lemma
applied to the long exact sequences in homotopy for the
two bundles.
\end{proof}

Lemma~\ref{lem: CD for Borel} gives us a convenient way to restate 
part of the
results of proposition~\ref{pro: configsphere}.

\begin{prop}
\label{pro: BSKpi1}
For $k \geq 3$ one has that 
$$
ESO(3) \times_{SO(3)} F(S^2,k) = K(PB_{k-3}(\mathbb{R}^2 - Q_2),1).
$$
Here we are using the natural action of $SO(3)$ on $F(S^2,k)$
described before. (We are using the convention, that $PB_0(X)$ denotes
the trivial group.) 
\end{prop}
\begin{proof}
Fix $k \geq 3$, then from proposition~\ref{pro: configsphere}, 
one has a $SO(3)$-equivariant
homotopy equivalence 
$$
\lambda: SO(3) \times Y \rightarrow F(S^2,k). 
$$ 
where we have denoted $F(S^2-Q_3,k-3)$ by $Y$ for convenience.
(This means that $Y$ is a point when $k=3$.)

Applying lemma~\ref{lem: CD for Borel} using $Id: SO(3) \rightarrow
SO(3)$ as $\mu$, and $\lambda$ as the equivariant map, one sees immediately
that since $B(\mu)$ and $\lambda$ are homotopy equivalences, then
$$
E(Id) \bar{\times} \lambda : ESO(3) \times_{SO(3)} (SO(3) \times
Y ) \rightarrow ESO(3) \times_{SO(3)} F(S^2,k)
$$
is a (weak) homotopy equivalence. However,
\begin{align*}
\begin{split}
ESO(3) \times_{SO(3)} (SO(3) \times Y)
&= (ESO(3) \times_{SO(3)} SO(3)) \times Y \\
&= ESO(3) \times Y \\
&\simeq Y 
\end{split}
\end{align*}
by remark~\ref{rem: basborel} and the contractibility of $ESO(3)$.
Theorem~\ref{thm: Kpi1} then gives us the desired result.
\end{proof}

Now notice the following. If $X$ is a left $G$-space, then
$F(X,k)$ is also a $G$-space via a diagonal action (as any group
action takes distinct points to distinct points). On the other
hand $\Sigma_k$ acts on $F(X,k)$ on the right by permuting
coordinates. It is easy to see that the $\Sigma_k$-action
on $F(X,k)$ commutes with the $G$-action. Thus if we look
at the $\Sigma_k$ action on $EG \times F(X,k)$ where
$\Sigma_k$ acts purely on the right factor by permuting coordinates,
then it is easy to check that this action descends to an
action of $\Sigma_k$ on the Borel construction $EG \times_G F(X,k)$.
It is also routine to check that this action of $\Sigma_k$ on
the Borel contruction is still free and that
$$
(EG \times_G F(X,k))/\Sigma_k = EG \times_G SF(X,k)
$$
where recall that $SF(X,k)$ denotes $F(X,k)/\Sigma_k$.

We summarize this useful fact in the following remark:

\begin{rem}
\label{rem: Borelcover}
Let $G$ be a topological group and let $X$ be a left $G$-space
then there is a natural $\Sigma_k$-covering map:
$$
EG \times_G F(X,k) \overset{\pi}{\rightarrow} EG \times_G SF(X,k)
$$
where the natural diagonal actions of $G$ on $F(X,k)$ and $SF(X,k)$
are used to form the Borel constructions. 
\end{rem}

Remark~\ref{rem: Borelcover} and proposition~\ref{pro: BSKpi1}
have the following immediate corollary:

\begin{cor}
\label{cor: BSKpi1}
$ESO(3) \times_{SO(3)} SF(S^2,k)$ is a $K(\pi,1)$ space for all $k \geq 3$.
\end{cor} 

One bad thing about the Borel constructions so far, is that although
they produced $K(\pi,1)$ spaces, these spaces did not have $PB_k(S^2)$
as their fundamental group.
We will now fix this point.

If a topological group $G$ has a universal cover $\tilde{G}$ (which
we always assume is connected), then it is elementary to show that
$\tilde{G}$ has the structure of a topological group such that
the covering map $\mu: \tilde{G} \rightarrow G$ is a homomorphism
of topological groups. $\tilde{G}$ is called the covering group
of $G$.

Now given a (left) $G$-space $X$, $X$ can also be viewed as a
$\tilde{G}$-space via $\mu$. Thus the identity map of $X$ is
$\mu$-equivariant in our previous terminology.
It then follows from lemma~\ref{lem: CD for Borel} that we have
the following commutative diagram where each row is a fibration:

$$
\begin{CD}
X @>>> E\tilde{G} \times_{\tilde{G}} X @>>> B\tilde{G} \\
@VV Id V @VV E(\mu) \bar{\times} Id V @VV B(\mu) V \\
X @>>> EG \times_G X @>>> BG 
\end{CD}
$$

We wish to study the vertical map in the middle.
To do this, we need to look at the vertical maps on the sides
first. It is obvious that the identity map is a homotopy equivalence,
 so let
us first look at $B(\mu)$. Recall the following elementary remarks:

\begin{rem}
For any topological group $G$, $BG$ is always path connected and
$\pi_n(BG) \cong \pi_{n-1}(G)$ for $n>1$.
\end{rem}
\begin{proof}
We have the universal principal $G$-bundle $EG \rightarrow BG$.
The result follows immediately from the long exact sequence
in homotopy for this fibration and the contractibility of $EG$.
\end{proof} 

\begin{rem}
A covering map $\pi: \tilde{X} \rightarrow X$ induces
isomorphisms between $\pi_n(\tilde{X})$ and $\pi_n(X)$
for all $n>1$. 
\end{rem}

It follows from these remarks that $\mu : \tilde{G} \rightarrow G$
induces isomorphisms between the higher homotopy groups ($\pi_n$ for
$n>1$) and hence that $B(\mu)$ induces isomorphisms in $\pi_n$ for
all $n \neq 2$. (Here recall that we insist that $\tilde{G}$ and
hence also $G$ is 
path connected.) Of course it induces the zero map on the $\pi_2$ level
as 
$$
\pi_2(B\tilde{G}) = \pi_1(\tilde{G}) = 0.
$$

With this knowledge of $B(\mu)$ and by
using the usual argument of applying the five-lemma to the
vertical maps between the long exact 
sequences in homotopy for the two fibrations, one finds easily that
$E(\mu) \bar{\times} Id$ induces an isomorphism in $\pi_n$
for all $n \neq 1,2$. Using the refined form of the five-lemma
(stated in lemma~\ref{lem: 5lem}), it also follows easily that
$E(\mu) \bar{\times} Id$ induces a monomorphism on the $\pi_2$
level. 

This is the main step in proving the following useful lemma:

\begin{lem}
\label{lem: compareBorel}
Let $G$ be a path connected, topological group.
Let $\tilde{G}$ be the universal covering group with
$\mu: \tilde{G} \rightarrow G$ the covering map. Furthermore
suppose $X$ is a $G$-space. Then
$E(\mu) \bar{\times} Id$ induces isomorphisms
$$
\pi_n(E\tilde{G} \times_{\tilde{G}} X) \cong \pi_n(EG \times_G X)
$$
for all $n \neq 1,2$ and a monomorphism
$$
\pi_2(E\tilde{G} \times_{\tilde{G}} X) \overset{1-1}{\rightarrow}
 \pi_2(EG \times_G X).
$$
Furthermore,
$$
\pi_1(E\tilde{G} \times_{\tilde{G}} X) \cong \pi_1(X) 
$$
\end{lem}
\begin{proof}
We have already proved everything besides the final isomorphism listed.
To prove this, look at the fiber bundle
$$
X \rightarrow E\tilde{G} \times_{\tilde{G}} X \rightarrow B\tilde{G}.
$$ 
Now we saw, in the paragraph preceding the lemma, that 
$\pi_2(B\tilde{G})=0$, but one also has
$$
\pi_1(B\tilde{G}) \cong \pi_{0}(\tilde{G})=0.
$$
Using this, it is easy to obtain our desired result 
from the
long exact sequence in homotopy for the fiber bundle above.
\end{proof}

Now it is a well known fact that $\widetilde{SO(3)}$ is $S^3$,
the group of unit quarternions, which is topologically a 3-sphere.
Thus applying lemma~\ref{lem: compareBorel} to this group and
using proposition~\ref{pro: BSKpi1} and
corollary~\ref{cor: BSKpi1}, one easily obtains:

\begin{cor}
$$
ES^3 \times_{S^3} F(S^2,k) = K(PB_k(S^2),1)
$$
and 
$$
ES^3 \times_{S^3} SF(S^2,k) = K(B_k(S^2),1),
$$
for $k \geq 3$.
\end{cor}

Thus, once again we see that the configuration space provides
a good model (this time using a Borel construction)
for the (pure) braid group of the underlying space.

Thus the only surface whose configuration space we have not studied
yet is $\mathbb{R}P^2$. We will do this now, in the next section.  

\section{Configuration spaces for the real projective plane}

\subsection{Orbit configuration spaces}

In our study of the configuration spaces of $\mathbb{R}P^2$,
we will find the concept of an orbit configuration space quite
useful. These were introduced by M. Xicot$\acute{e}$ncatl and
are defined as follows:

\begin{defn}
Let $G$ be a group and $X$ a $G$-space. Then 
$$
F_G(X,k) = \{(x_1,\dots, x_k) \in X^k | x_i,x_j \text{ are in different
 G orbits if } i \neq j \}.
$$
$F_G(X,k)$ is called the $k$-fold orbit configuration space of $X$.
\end{defn}

\begin{rem}
Thus we see that $F_G(X,k)$ is the space of $k$-tuples of points 
which lie in distinct $G$ orbits of $X$.
Note that $F_G(X,1)=X$ and $F_1(X,k)=F(X,k)$ for the action of
the trivial group 1 on $X$.
\end{rem}

We will now prove some elementary but useful properties of these orbit
configuration spaces. First some necessary definitions.

\begin{defn}
Given two groups $G_1$ and $G_2$, 
a left $G_1$-space $X$ is said to have a compatible $G_2$ action if \\
(a) $X$ is a left $G_2$-space. \\
(b) For every $g_1 \in G_1, g_2 \in G_2, x \in X$, there
exists $g_1' \in G_1$ such that
$$
g_2 \cdot (g_1 \cdot x) = g_1' \cdot (g_2 \cdot x).
$$ 
If it is always possible to take $g_1'=g_1$, we say that the two actions
commute.
\end{defn}

\begin{rem}
\label{rem: compatorbit}
Notice that if a $G_1$-space $X$ has a compatible $G_2$ action
then the $G_2$ action takes $G_1$-orbits to $G_1$-orbits
and hence induces a $G_2$ action on $X/G_1$. (To guarantee the continuity 
of this action, one should assume that $G_2$ is locally compact, Hausdorff 
so that $G_2 \times X \rightarrow G_2 \times X/G_1$ is a quotient map. This 
holds for all Lie groups and hence discrete groups.)
\end{rem}

\begin{rem}
The most important examples of compatible actions are: \\
(a) commuting actions. \\
(b) When $G_1=G_2=G$ and both groups act in the same way on $X$.
These are compatible actions as
$$
g_2 \cdot (g_1 \cdot x) = (g_2 g_1 g_2^{-1}) \cdot (g_2 \cdot x)
$$
and so we can take $g_1'=g_2 g_1 g_2^{-1}$.
\end{rem}

\begin{lem}
\label{lem: actiononorbit}
If $X$ is a $G_1$-space which supports a compatible $G_2$-action,
then there is a natural $G_2^k$ action on $F_{G_1}(X,k)$ defined via
$$
(g_1,\dots,g_k) \cdot (x_1,\dots,x_k) = (g_1\cdot x_1,\dots,g_k\cdot x_k)
$$
for all $(g_1,\dots,g_k) \in G_2^k$ and $(x_1,\dots,x_k) \in F_{G_1}(X,k)$.
Notice of course, this also gives a $G_2$ action on $F_{G_1}(X,k)$ obtained
by precomposing the above action with the diagonal homomorphism 
$\Delta : G_2 \rightarrow G_2^k$ which sends $g$ to $(g,\cdots,g)$.
\end{lem}
\begin{proof}
The proof follows easily from remark~\ref{rem: compatorbit} and
is left to the reader.
\end{proof}

>From lemma~\ref{lem: actiononorbit}, we see that there is always a natural
action of $G^k$ on $F_G(X,k)$. If the original $G$-action on $X$ is
reasonable, we can say something about this $G^k$-action on $F_G(X,k)$: 

\begin{prop}
\label{pro: orbittonormalconfig}
If $\pi: X \rightarrow X/G$ is a principal $G$-bundle, 
there is a map
$$
\bar{\pi}: F_G(X,k) \rightarrow F(X/G,k)
$$
which is a principal $G^k$-bundle, where 
$$
\bar{\pi}(x_1,\dots,x_k)=
(\pi(x_1),\dots,\pi(x_k))
$$ 
and the $G^k$ action on $F_G(X,k)$
is that which is described in lemma~\ref{lem: actiononorbit}.
\end{prop}
\begin{proof}
First notice that by taking a $k$-fold product we get a principal
$G^k$-bundle 
$$
G^k \rightarrow X^k \overset{\times \pi}{\rightarrow} (X/G)^k.
$$
Then we can take the pullback of this bundle with respect to the inclusion map
$i: F(X/G,k) \rightarrow (X/G)^k$ and hence get the commutative diagram:
$$
\begin{CD}
G^k @>> Id > G^k \\
@VVV @VVV \\
E @>>> X^k \\
@VV \bar{\pi} V @VV \times \pi V \\
F(X/G,k) @>> i > (X/G)^k
\end{CD}
$$
Thus the left hand column is also a principal $G^k$-bundle and
the reader can easily check from the definition of a pullback
(see \cite{Bred}),
 that $E$ is naturally identified with  
$F_G(X,k)$ and that with this identification, $\bar{\pi}$
has the form stated above and the $G^k$ action on $E=F_G(X,k)$ is that
which is described in lemma~\ref{lem: actiononorbit}.
\end{proof}

\begin{rem}
Recall, that when $G$ is discrete, a principal $G$-bundle is the same thing
as a regular cover with covering group $G$.
With this, it follows easily from proposition~\ref{pro: orbittonormalconfig},
that if $G$ is a finite group and $X$ a free $G$-space, then there is a
covering map
$$
\pi: F_G(X,k) \rightarrow F(X/G,k)
$$
with covering group $G^k$. 
\end{rem}

Proposition~\ref{pro: orbittonormalconfig} shows us immediately how
orbit configuration spaces can help us understand the configuration
space of $\mathbb{R}P^2$. To see this, we can take 
$S^2$ as a $\mathbb{Z}/2\mathbb{Z}$-space where the nontrivial 
element 
of $\mathbb{Z}/2\mathbb{Z}$ is acting via the
antipodal map. Then proposition~\ref{pro: orbittonormalconfig}
gives us the following covering map
$$
(\mathbb{Z}/2\mathbb{Z})^k \rightarrow 
F_{\mathbb{Z}/2\mathbb{Z}}(S^2,k) \overset{\pi}{\rightarrow} 
F(\mathbb{R}P^2,k).
$$

Thus we need to study $F_{\mathbb{Z}/2\mathbb{Z}}(S^2,k)$. To do
this, we will need to use the analogue of the Fadell-Neuwirth theorem
(theorem~\ref{thm: fib}) for orbit configuration spaces. We will
state and prove this analogue next.

\begin{defn}
Given a free $G$-space $X$, 
we will use the notation $O_k$ to denote the disjoint union
of $k$ distinct $G$-orbits. Notice $X-O_k$ is hence also
a $G$-space.
\end{defn}

\begin{thm}
Let $G$ be a finite group and $X$ be a free $G$-space. Suppose
further that $X$ is a connected manifold without boundary with
dimension $\geq 1$.
Then 
for all $n > k$, there
are fibrations
$$
F_G(X-O_k,n-k) \rightarrow F_G(X,n) \overset{\pi}{\rightarrow} F_G(X,k)
$$
where $\pi$ is projection onto the first $k$ factors.
\end{thm}
\label{thm: orbit fib}
\begin{proof}
Once one has verified that $\pi$ is a fibration, it is an easy exercise
to see that the fiber is as described in the statement of this theorem.
Thus, we will only show that $\pi$ is a fibration. 

Since the composition
of fibrations is a fibration, it is enough to show that $\pi$ is
a fibration in the case when $n=k+1$, since an arbitrary projection
from $F_G(X,n)$ to $F_G(X,k)$ can be broken up as a composition of maps
of this sort. Thus, we assume that $n=k+1$ from now on. 

Let the order of $G$ be $s$ and write $G=\{g_1,\dots,g_s\}$. 
Define $\lambda: F_G(X,k) \rightarrow F(X,sk)$ by 
$$
\lambda(x_1,\dots,x_k) = 
(g_1x_1,g_2x_1,\dots,g_sx_1,g_1x_2,\dots,g_sx_2,\dots,g_1x_k,\dots,g_sx_k).
$$

The following diagram commutes:

$$
\begin{CD}
F(M-Q_{ks},1) @>>> F(M, ks+1) @> \mu >> F(M,ks) \\ 
@AA Id A @AA \bar{\lambda} A @AA \lambda A \\
F_G(M-O_k,1) @>>> E @> \bar{\mu} >> F_G(M,k)
\end{CD}
$$
Here the top row is a fibration by theorem~\ref{thm: fib}, 
where $\mu$ is projection onto the
first $ks$ factors. The bottom row is the pullback of this fibration
under
the map $\lambda$. 

Thus $\bar{\mu}$ is a fibration. It remains only to show that the pullback
$E$
is naturally homeomorphic to $F_G(M,k+1)$. Recall that 
$$
E=\{(x,y) \in F_G(M,k) \times F(M,ks+1) | \lambda(x)=\mu(y)\}.
$$

Define $\theta: F_G(M,k+1) \rightarrow F_G(M,k) \times F(M,ks+1)$
by
$$
\theta(x_1,\dots,x_{k+1})=
((x_1,\dots,x_k), (g_1x_1,\dots,g_sx_1,\dots,g_1x_k,\dots,g_sx_k,x_{k+1})).
$$
It is easy to check that in fact $\theta: F_G(M,k+1) \rightarrow E$.

Now, $\theta: F_G(M,k+1) \rightarrow E$ has an inverse given by
$$
((x_1,\dots,x_k),(g_1x_1,\dots,g_sx_k,x_{k+1}))
\rightarrow
(x_1,\dots,x_k,x_{k+1}).
$$

Thus, $\theta$ is our desired homeomorphism and furthermore
$\bar{\mu} \circ \theta: F_G(X,k+1) \rightarrow F_G(X,k)$ is indeed 
projection onto the first $k$ coordinates.
This is a fibration as $\bar{\mu}$ is. Thus the theorem is proven. 
\end{proof}

(The authors would like to thank Dan Cohen for informing us about
the nice proof above. Previous proofs consisted of going through the
proof of theorem~\ref{thm: fib} carefully and doing the necessary
modifications, which is considerably more messy.)

\subsection{Application of orbit configuration spaces to 
$F(\RR P^2,k)$}

First, we need to study $F_{\mathbb{Z}/2\mathbb{Z}}(S^2,k)$.

\begin{lem}
\label{lem: orbitKpi1}
Let $S^2$ be given a $\mathbb{Z}/2\mathbb{Z}$ action
via the antipodal map. 

Then $F_{\mathbb{Z}/2\mathbb{Z}}(S^2 - O_n,k)$ is a $K(\pi,1)$
space for all $n,k \geq 1$. 

Furthermore $\pi_1(F_{\mathbb{Z}/2\mathbb{Z}}(S^2 - O_n,k))$ is
a polyfree group of cohomological dimension $k$ and it has
a polyfree series with exponents 
$$
(2(n+k-1)-1,\dots,2(n+1)-1,2(n+0)-1).
$$
\end{lem}
\begin{proof}
We will prove this lemma by induction on $k$. First the case when
$k=1$. Here we have that
$$
F_{\mathbb{Z}/2\mathbb{Z}}(S^2-O_n,1) = S^2-O_n = \mathbb{R}^2 - Q_{2n-1}
= K(F_{2n-1},1).
$$
Thus the result follows as $F_{2n-1}$, the free group on $2n-1$ generators,
is polyfree, of cohomological dimension one with the stated exponent.

So without loss of generality, $k>1$ and we can assume the theorem
is proven for numbers smaller than $k$.

Now by theorem~\ref{thm: orbit fib}, we have a fibration
$$
F_{\mathbb{Z}/2\mathbb{Z}}(S^2-O_{n+k-1},1) \rightarrow
F_{\mathbb{Z}/2\mathbb{Z}}(S^2-O_n,k) \rightarrow
F_{\mathbb{Z}/2\mathbb{Z}}(S^2-O_n,k-1).
$$
By induction, the fiber and the base space are $K(\pi,1)$ spaces
and hence, by the long exact sequence in homotopy, so is the total space.  

Furthermore, on the level of fundamental groups, the long exact
sequence in homotopy for the fibration above gives us a short exact
sequence of groups:

$$
\pi_1(F_{\mathbb{Z}/2\mathbb{Z}}(S^2-O_{n+k-1},1)) \rightarrow
\pi_1(F_{\mathbb{Z}/2\mathbb{Z}}(S^2-O_n,k)) \rightarrow
\pi_1(F_{\mathbb{Z}/2\mathbb{Z}}(S^2-O_n,k-1)).
$$

By induction, the base (quotient) group is polyfree, of cohomological 
dimension $k-1$,
with exponents:
$$
(2(n+k-2)-1,\dots,2(n+0)-1)
$$
and also the fiber group (kernel) is a free group of rank $2(n+k-1)-1$.
>From this, the properties of $\pi_1(F_{\mathbb{Z}/2\mathbb{Z}}(S^2-O_n,k))$
stated in the lemma, follow easily.
\end{proof}

Unfortunately, lemma~\ref{lem: orbitKpi1} does not tell us anything
about $F_{\Ztwo}(S^2,k)$ directly - the lemma only applies if at least one 
$\Ztwo$-orbit
has been removed from $S^2$. We have to do some genuine geometric
analysis to go any further, so let us analyze $F_{\Ztwo}(S^2,2)$
and show it is homotopy equivalent to $SO(3)$ in a nice way.

First, let us notice that if we view $S^2$ as the unit vectors
in $\RR^3$, then $F_{\Ztwo}(S^2,2)$ is naturally identified as 
the space of pairs of linearly independent unit vectors of $\RR^3$.  

Thus, inside of $F_{\Ztwo}(S^2,2)$ is the subspace
$$
U = \{(x,y) | x \text{ is orthogonal to } y \text{ and } x,y \in S^2\}
$$
of pairs of orthogonal unit vectors. Notice that $U$ can be 
naturally identified with
the unit tangent bundle of $S^2$ (We will not need this fact).

$SO(3)$ acts naturally on $S^2$ and this action commutes with the
antipodal map, thus by lemma~\ref{lem: actiononorbit}, $SO(3)$
acts diagonally on $F_{\Ztwo}(S^2,2)$.

It is easy to see that this action preserves the subspace $U$.
In fact, even more is true. Let $e_i$ be the ith standard 
unit vector of $\RR^3$
for $i=1,2,3$. Then given $(x,y) \in U$, since there are elements of
$O(3)$ mapping $(e_1,e_2,e_3)$ to any other given orthonormal basis, 
it is easy to see
that there is an element of $SO(3)$ mapping $(e_1,e_2)$ to $(x,y)$.
In fact this element is unique as any element of $SO(3)$
mapping $(e_1,e_2)$ to $(x,y)$ would have to map $e_1 \times e_2 = e_3$
to $x \times y$ and so is determined on a basis. (Here $\times$ stands
for the cross product in $\RR^3$.)

Thus we see that the $SO(3)$ action on $U$ is transitive and free
and hence $U=SO(3)$. Furthermore, the $SO(3)$ action on $U$ corresponds
to the left multiplication action on $SO(3)$ under this correspondence. 

\begin{prop}
\label{pro: gramschmidtS^2}
The inclusion $i: U \rightarrow F_{\Ztwo}(S^2,2)$ is a 
$SO(3)$-equivariant homotopy
equivalence. Furthermore $U=SO(3)$, with $SO(3)$-action given
by left multiplication.
\end{prop}
\begin{proof}
We have already proven the final sentence of the proposition
 in the preceding paragraph.
It remains only to show that $i$ is a homotopy equivalence. 

We will need a bit of notation. For $x,y \in \RR^n$, let $x \cdot y$ denote
the dot product of $x$ and $y$. Let $\parallel x \parallel = 
\sqrt{x \cdot x}$ denote the norm of $x$.

Now we construct $R: F_{\Ztwo}(S^2,2) \times I \rightarrow F_{\Ztwo}(S^2,2)$
as follows:
$$
R(x,y,t) = (x, \frac{y-t(x \cdot y) x}{\parallel y-t(x \cdot y)x \parallel}).
$$
Notice $R$ is well defined as $y$ is not a multiple of $x$.
The following properties of $R$ follow easily: \\
(a) $R(x,y,0) = (x,y) \text{ for all } (x,y) \in F_{\Ztwo}(S^2,2)$, \\
(b) $R(x,y,1) \in U$, \\
(c) $R(x,y,t) = (x,y) \text{ for all } (x,y) \in U, t \in I$. \\
(d) $R(\alpha x, \alpha y, t) = \alpha R(x,y,t)$ for all $(x,y) \in 
F_{\Ztwo}(S^2,2), t \in I, \alpha \in O(3)$. Here $O(3)$ is acting
diagonally on $F_{\Ztwo}(S^2,2)$ in the natural way.

Properties (a)-(c) tell us that $U$ is a strong deformation retract
of $F_{\Ztwo}(S^2,2)$ which is all we need to finish the proof
of the proposition. Property (d) says in fact, that there is a
strong deformation through $SO(3)$-equivariant maps of $F_{\Ztwo}(S^2,2)$
to itself.

(Notice, that $R$, is in effect, performing a Gram-Schmidt orthogonalization
process globally.)
\end{proof}

\begin{cor}
\label{cor: orbS^22}
If $F_{\Ztwo}(S^2,2)$ is given the diagonal $SO(3)$ action induced
from the natural action of $SO(3)$ on $S^2$ then 
$ESO(3) \times_{SO(3)} F_{\Ztwo}(S^2,2)$ is (weakly) contractible.
\end{cor}
\begin{proof}
Proposition~\ref{pro: gramschmidtS^2} shows that 
$i: U=SO(3) \rightarrow F_{\Ztwo}(S^2,2)$ is a $SO(3)$-equivariant
homotopy equivalence. 
By lemma~\ref{lem: CD for Borel}, it follows that
$$
Id \bar{\times} i: ESO(3) \times_{SO(3)} U \rightarrow
ESO(3) \times_{SO(3)} F_{\Ztwo}(S^2,2)
$$
is a (weak) homotopy equivalence.

On the other hand,
$$
	ESO(3) \times_{SO(3)} U = ESO(3) \times_{SO(3)} SO(3) 
= ESO(3) \simeq * 
$$

\end{proof}

This was the crucial ingredient to the following proposition:

\begin{prop}
\label{pro: orbitS2Kpi1}
If $F_{\Ztwo}(S^2,k)$ is given the diagonal $SO(3)$-action induced
from the natural $SO(3)$ action on $S^2$ then
$ESO(3) \times_{SO(3)} F_{\Ztwo}(S^2,k)$ is a $K(\pi,1)$-space
for all $k \geq 2$.
\end{prop}
\begin{proof}
We will prove this by induction on $k$. The case $k=2$ was
proven in corollary~\ref{cor: orbS^22} so assume $k>2$ and
that we have proven the proposition for smaller numbers.

Let $\pi: F_{\Ztwo}(S^2,k) \rightarrow F_{\Ztwo}(S^2,2)$
be projection onto the first $2$ factors. Notice that $\pi$
is $SO(3)$-equivariant. 

By 
theorem~\ref{thm: orbit fib}, $\pi$ is a fibration with fiber
$F_{\Ztwo}(S^2-O_2,k-2)$ which is a $K(\pi,1)$-space by 
lemma~\ref{lem: orbitKpi1}.
Thus $\pi$ induces an isomorphism in $\pi_i$ for $i \neq 1,2$. 
However $\pi_2(F_{\Ztwo}(S^2,2))=\pi_2(SO(3))=0$ so in fact
in the long exact sequence for the fibration above, the boundary
map from $\pi_2(F_{\Ztwo}(S^2,2))$ to $\pi_1(fiber)$ necessarily
vanishes and $\pi$ is an isomorphism on the $\pi_2$ level also.

By lemma~\ref{lem: CD for Borel}, we have the following commutative
diagram
$$
\begin{CD}
F_{\Ztwo}(S^2,k) @>>> ESO(3) \times_{SO(3)} F_{\Ztwo}(S^2,k) @>>>
BSO(3) \\
@VV \pi V @VV Id \bar{\times} \pi V @VV Id V \\
F_{\Ztwo}(S^2,2) @>>> ESO(3) \times_{SO(3)} F_{\Ztwo}(S^2,2) @>>>
BSO(3)
\end{CD}
$$
where each row is a fibration.
In the preceding paragraph, we saw that $\pi$ induced isomorphisms
in $\pi_i$ for all $i \neq 1$. Thus by the refined form of the
five lemma given in lemma~\ref{lem: 5lem}, we see
that $Id \bar{\times} \pi$ induces a monomorphism in $\pi_i$ for
all $i \neq 1$. Since $\pi_i(ESO(3) \times_{SO(3)} F_{\Ztwo}(S^2,2))=0$
for all $i$ this lets us conclude that
$\pi_i(ESO(3) \times_{SO(3)} F_{\Ztwo}(S^2,k))=0$ for all $i \neq 1$
which is what we set out to prove.
\end{proof}

Now we are finally ready to study the configuration space
$F(\RR P^2,k)$. Recall we had the covering map:
$$
(\Ztwo)^k \rightarrow F_{\Ztwo}(S^2,k) \overset{\pi}{\rightarrow} F(\RR P^2,k).
$$
As before, $SO(3)$ acts diagonally on $F_{\Ztwo}(S^2,k)$ using the
natural action of $SO(3)$ on $S^2$. Of course, $SO(3)$ also acts
naturally on $\RR P^2$ viewed as the space of lines in $\RR^3$.
Hence, $SO(3)$ acts diagonally on $F(\RR P^2,k)$.
Furthermore, it is easy to see, that the covering map $\pi$ above
is $SO(3)$-equivariant with respect to these actions. (For $k=1$, $\pi$
is just the map taking a unit vector, to the line it spans.) 

Being a covering map, $\pi$ induces isomorphisms in $\pi_i$ for all
$i \neq 1$, and a monomorphism in $\pi_1$.
We have the usual commutative diagram

$$
\begin{CD}
F_{\Ztwo}(S^2,k) @>>> ESO(3) \times_{SO(3)} F_{\Ztwo}(S^2,k)
@>>> BSO(3) \\
@VV \pi V @VV Id \bar{\times} \pi V @VV Id V \\
F(\RR P^2,k) @>>> ESO(3) \times_{SO(3)} F(\RR P^2,k) @>>> BSO(3)
\end{CD}
$$
where each row is a fibration.

>From the five lemma (lemma~\ref{lem: 5lem}), it follows easily
that $Id \bar{\times} \pi$ induces an epimorphism in $\pi_i$ for
$i \neq 1$. Thus from proposition~\ref{pro: orbitS2Kpi1}, it follows that
$ ESO(3) \times_{SO(3)} F(\RR P^2,k)$ is a $K(\pi,1)$-space. 
This is the result we wanted in this section and we state it as:

\begin{thm}
\label{thm: BorelconfRP2Kpi1}
$ESO(3) \times_{SO(3)} F(\RR P^2,k)$ is a $K(\pi,1)$-space if $k \geq 2$ and
furthermore
$$
ES^3 \times_{S^3} F(\RR P^2,k) = K(PB_k(\RR P^2),1) 
$$
and 
$$
ES^3 \times_{S^3} SF(\RR P^2,k) = K(B_k(\RR P^2),1).
$$ 
Here, the $SO(3)$ action is induced from the natural action of
$SO(3)$ on $\RR P^2$ viewed as the space of lines in $\RR^3$.
The $S^3$ action is obtained from the $SO(3)$ action using that
$S^3$ is the universal covering group of $SO(3)$. 
\end{thm}
\begin{proof}
Most of the theorem was proven in the preceding paragraph. The
only thing remaining is the statement about the $S^3$ Borel
constructions which follows immediately from lemma~\ref{lem: compareBorel}
and remark~\ref{rem: Borelcover}. 
\end{proof}

\begin{rem}
>From the fibration, 
$$
F(\RR P^2,k) \rightarrow ES^3 \times_{S^3} F(\RR P^2,k) \rightarrow
BS^3
$$
and the fact that $ES^3 \times_{S^3} F(\RR P^2,k)$ is a $K(\pi,1)$-space,
we see easily that
$$
\pi_n(F(\RR P^2,k)) \cong \pi_{n+1}(BS^3) \cong \pi_{n}(S^3)
$$
for $n \geq 2$.

Thus, $F(\RR P^2,k)$ has the same higher homotopy as the 3-sphere.
\end{rem}

This completes our analysis of $F(M,k)$ for 2-manifolds $M$
without boundary.
The conclusion is that these are always $K(PB_k(M),1)$-spaces except
when $M$ is $S^2$ or $\RR P^2$. However in these cases,
there is an associated Borel construction which is a $K(PB_k(M),1)$.

For now, these results might seem to be a pretty formal analysis 
of the homotopy type of the configuration spaces of surfaces - however we
will see their true power when we begin talking about labeled
configuration spaces, and
the ``classical'' connection of these labeled configuration spaces with
loop spaces.

Furthermore, Borel constructions will provide a crucial tool in connecting
the braid groups we have studied, to mapping class groups, another important
class of groups associated to surfaces, which we will introduce
and study later on in these notes.

\section{Mapping class groups}

\vskip .2in

Let $M$ denote an orientable surface and let $Top(M)$ be the group of
orientation preserving homeomorphisms of M. We make it a 
topological group by giving it the compact open topology. 

Let $Q_k$ be a 
set of $k$ distinct points in $M$ and let $Top(M, Q_k)$ be the 
topological subgroup 
of $Top(M)$ which leaves the set $Q_k$ invariant. 

Recall the standard isotopy lemma, (see \cite{GP})

\begin{lem}[Isotopy Lemma]
\label{lem: isotopy}
Let $M$ be a connected smooth manifold of dimension strictly bigger 
than one. Then for $(x_1,\dots,x_k), (y_1,\dots,y_k) \in F(M,k)$, 
there exists a diffeomorphism $\phi: M \to M$ such that 
$\phi$ is isotopic to the identity and $\phi(x_i)=(y_i)$ for 
$1 \leq i \leq k$.
\end{lem}
 
From the isotopy lemma, there is an orientation preserving 
homeomorphism of $M$ carrying one set of $k$ distinct points 
$Q_k$ to any other set of $k$ distinct points $Q_k'$, 
and thus $Top(M,Q_k)$ is conjugate to $Top(M,Q_k')$ in $Top(M)$
and so we will sometimes write $Top(M,k)$ when the points are understood.

We will also look at $PTop(M,Q_k)$, the topological subgroup of $Top(M,Q_k)$ 
consisting of homeomorphisms which actually fixe the points of $Q_k$ 
pointwise. Thus $PTop(M,Q_k)$ is the kernel of the natural 
homomorphism from $Top(M,Q_k) \to \Sigma_k$ induced by 
sending a homeomorphism leaving the set $Q_k$ invariant to the 
permutation it induces on those points. Again, we will write 
$PTop(M,k)$ when the points $Q_k$ are understood.

Recall that if $M$ is
a surface then the inclusion of the group of orientation
preserving diffeomorphisms $Diff^+(M) \to\ Top(M)$ is a homotopy
equivalence. Some of the motivation for
considering $PTop(M,k)$ is described next.

\vskip .2in

Properties of these subgroups frequently correspond to a lifting
question relating branched covering spaces. Namely given a
branched cover $N \to M$ which is branched over a finite set $Q_k$, 
information about $Top(M,k)$ gives information about
$Top(N)$. That is, homeomorphisms of $M$ which leave the branch set invariant
lift (in possibly more than one way) to homeomorphisms of $N$.

\vskip .2in

One direction of these notes is to obtain information
concerning $Top(M,k)$ in order to obtain information about
$Top(N)$. This approach is carried out in a few cases in these
notes for the specific examples of surfaces of genus 0,1, and 2.
This approach dates back to Hurwitz, and has been exploited in
\cite{BH}; these classical methods have proven to be
useful also for cohomological analysis.

\vskip .2in

In a few of the applications, pointed versions of these
constructions are useful. For example, the cohomology of the
"pointed mapping class group" for a genus 1 surface with marked
points (to be defined below) has a clean cohomological
description. The analogous description without the assumption of
"pointed maps" has a technically more complicated description.
This structure is analogous to the behavior of certain function
spaces. That is, the space of pointed maps from a circle to a space
X, the loop space $\Omega X$ is frequently accessible from a
homological point of view. At the same time, the space of free
maps from a circle to X, the free loop space $\Lambda X$, 
is frequently more delicate.

Mapping class groups, in some cases, reflect similar behavior as
we will see later when we study the
pointed mapping class group for punctured copies of genus zero, one,
and two surfaces.

\vskip .2in
Recall given a topological group $G$, the path components $pi_0(G)$ 
form a (discrete) group where the group structure is induced from $G$.

\begin{defn}
Let $M$ be a closed orientable surface of genus $g$. 
The mapping class group 
$\Gamma_g^k=\pi_0(Top(M,Q_k))$. 

The pure mapping class group $P\Gamma_g^k$ is
the kernel of the natural homomorphism $\Gamma_g^k  \to\
\Sigma_k$, induces by sending a homeomorphism to the permutation 
it induces a the points $Q_k$.
\end{defn}

\begin{defn}
Let $M$ be a closed orientable surface of genus $g$, $Q_k$ a 
set of $k$ distinct points of $M$ and $p$ a fixed point in $M-Q_k$.

The pointed mapping class group 
$\Gamma_g^{k,1}$ is the group of
path components of the orientation preserving homeomorphisms which
(1) preserve the point $p$, and \\
\noindent
(2) leave the set $Q_k$ invariant. 

The pure pointed
mapping class group $P\Gamma_g^{k,1}$ is the kernel of the natural
homomorphism $\Gamma_g^{k,1}  \to\ \Sigma_k$.
\end{defn}

\vskip .2in

The group $Top(M)$ acts on the configuration space of points in $M$,
$F(M,k)$, diagonally. A "folk theorem" that has been useful gives
(1) there are natural $K(\pi,1)'s$ obtained from the associated
Borel construction (homotopy orbit spaces) for groups acting on
configuration spaces, and \\
\noindent
(2) these configuration spaces are
analogous to homogeneous spaces in the sense that they are
frequently homeomorphic to a quotient of a topological group by a
closed subgroup.

\vskip .2in

Namely, let $G$ be a subgroup of $Top(M)$, and consider the diagonal
action of $G$ on $F(M,k)$ together with the homotopy orbit spaces

$$
EG \times_G F(M,k)
$$ 
and
$$
EG \times_G {F(M,k)/ \Sigma_k}.
$$

\vskip .2in

\begin{lem}
\label{lem: closedsubgroup}
Assume that $M$ is a nonempty manifold of dimension at least 1. Then
$Top(M,k)$ and $PTop(M,k)$ are closed subgroups of $Top(M)$.
\end{lem}
\begin{proof}
Given any point $f$ in the complement of $Top(M,k)$ in $Top(M)$,
there is at least one $m_i \in Q_k$ that is taken to a point outside 
$Q_k$.
The open set of continuous functions 
that carry $m_i$ to $M-Q_k$ in the complement of
$Top(M,k)$ in $Top(M)$ is a neighbourhood of $f$. 
Thus it follows that the complement of $Top(M,k)$ in $Top(M)$ 
is open and the lemma follows for $Top(M,k)$. 
The proof for $PTop(M,k)$ is similar.
\end{proof}

\vskip .2in

Consider the natural evaluation map $Top(M)  \to\  F(M,k)$ given
by $Top(M)$ acting on the point $(m_1, m_2, ... , m_k)$. 
Since the isotropy group of the $Top(M)$ action on $F(M,k)$ 
at $(m_1,\dots,m_k)$ 
is the group $Top(M,k)$, there is an induced map 
$$
\rho: Top(M)/Top(M,k)  \to\  F(M,k).
$$
By the isotopy lemma, $\rho$ is onto and hence a continuous bijection.

In the next theorem, we will show among other things, 
that in fact the map $\rho$ above 
is a homeomorphism. Thus configuration spaces of surfaces are 
``homogeneous spaces'' of suitable topological groups. 
(The quotation marks about the word homogeneous space is due 
to the fact that $Top(M)$ is only a topological group, not a 
Lie group). 

The theorem will be proven from a construction in Steenrod's
book "The topology of fibre bundles". Namely, consider 
$$ 
H \to\ G \to\  G/H 
$$ 
where $H$ is a closed subgroup of $G$, and the map $G
\to\  G/H$ admits "local cross-sections".  Then the induced map 
$BH \to\ BG$ is
the projection map in a fibre bundle with fibre given by the space
of left cosets $G/H$. [Steenrod, page 30]

The definition of "local cross-sections" is given as follows: Let
$H$ be a closed subgroup of $G$ with
natural quotient map $p: G \to\ G/H$ and let $x \in G/H$. 
A local cross-section of $p: G \to G/H$ at $x$ 
is a continuous function $f: V \to G$ where $V$ is an 
open neighborhood of $x$ in
$G/H$ satisfying $pf(x)=x$ for all $x \in V$. [Steenrod, page
30].

\begin{thm}
\label{thm: Topfibrations}
Assume that $M$ is an orientable surface without boundary.
Then
\begin{enumerate}
\item There is a principal fibration
$$
Top(M,k)  \to\ Top(M)  \to\  Top(M)/Top(M,k).
$$
\item The map 
$$
\rho: Top(M)/Top(M,k)  \to\  F(M,k)
$$
is a homeomorphism.
\item The homotopy theoretic fibre of the natural map
$BTop(M,k)  \to\ BTop(M)$ is $F(M,k)$, and $ETop(M) \times_{Top(M)} F(M,k)$
is homotopy equivalent to $BTop(M,k)$.
\end{enumerate}
\end{thm}

\vskip .2in

The proof of Theorem~\ref{thm: Topfibrations} 
depends on the next lemma. Here, let
$D^n$ denote the standard $n$-disk, i.e., 
the points in $\RR^n$ of euclidean norm at
most $1$ with interior denoted $D^{o,n}$ with $(0,0,...,0)$ the
origin in $D^n$. The map $\theta$ in the next lemma was quite
useful in the article by Fadell and Neuwirth \cite{FN} while the formula
here was written explicitly in \cite{X}.

\vskip .2in

\begin{lem}
\label{lem: diskmap}
\begin{enumerate}

\item There is a continuous map

$ \theta: D^{o,n} \times D^n \to\  D^n $

such that $ \theta(x,-)$ fixes the boundary of  $D^n$ pointwise,
and $ \theta(x,x) = (0,0,...,0)$ for every x in $D^{o,n}$.

\item If $M$ is a surface without boundary, then
there exists a basis of open sets $U$ for the topology of $F(M,k)$
together with local sections

$\phi: U  \to \ Top(M)$ such that the composite

$$ 
U  \overset{\phi}{\rightarrow} Top(M) \to\ Top(M)/Top(M,k)  
\overset{\rho}{\rightarrow}  F(M,k)
$$ 
is a homeomorphism onto $U$.

\item The natural map $\rho: Top(M)/Top(M,k)  \to\  F(M,k)$ is a
homeomorphism.
\end{enumerate}
\end{lem}
\begin{proof}

Define $ \alpha: D^{o,n} \to\ R^n $ by the formula $ \alpha(x) =
x/(1-|x|)$, and so $\alpha^{-1}(z) = z/(1+ |z|)$.

\vskip .2in

For a fixed element q in $ D^{o,n} $, define $\gamma_q: D^n \to\
D^n$ by the formula

\begin{enumerate}

\item  $\gamma_q(y) = y$ for y in the boundary of $D^n$, and

\item  $\gamma_q(y) = \alpha^{-1}( y/(1-||y||) - q/(1-|q|)$ for y
in $ D^{o,n} $.

\end{enumerate}

\vskip .2in

Define $ \theta: D^{o,n} \times D^n \to\  D^n $ by the formula $
\theta(q,y) = \gamma_q(y) $. Notice that $ \theta $ is continuous,
and $ \theta(q,q) = (0,0,...,0)$, and so part (1) of the lemma
follows.

\vskip .2in

To prove part (2), consider a point $(p_1, p_2, ... , p_k)$ in
$F(M,k)$ together with disjoint open discs  $ D^{o,2}(p_1),
D^{o,2}(p_2),....,D^{o,2}(p_k)$ where $ D^{o,2}(p_i)$ is a disc
with center $p_i $. ( There is a choice of homeomorphism in the
identification of each open disc with an open coordinate patch of
M; this choice will be suppressed here.) Let U be the open set in
$F(M,k)$ given by the product
$D^{o,2}(p_1) \times D^{o,2}(p_2)\times ...\times D^{o,2}(p_k)$.
The sets U give a basis for the topology of $F(M,k)$ which depend
on the choice of the discs $ D^{o,2}(p_i) $. Define $\phi: U  \to\
Top(M) $ by the formula $ \phi((y_1, y_2,...,y_k)) = H$ for $H$ in
$ Top(M) $ where $(y_1, y_2,...,y_k)$ is in $U = D^{o,2}(p_1)
\times D^{o,2}(p_2)\times ...\times D^{o,2}(p_k)$, and H is the
homeomorphism of M given as follows.

\begin{enumerate}

\item H(x) = x if x is in the complement in the union of the
$ \coprod_ {1 \leq i \leq k }D^{o,2}(p_i)$, and

\item  $ H(x)$ =  $\theta(p_i,x)$ if x is in $ D^{o,2}(p_i)$.

\end{enumerate}

\vskip .2in

Clearly $H$ is in $Top(M)$. Notice that $\phi$ is continuous if an
only if the adjoint of $\phi$, $adj(\phi): U \times M  \to\ M $ is
continuous as all spaces here are locally compact, and Hausdorff.
But then continuity follows at once from the first part of the
lemma  as $ adj(\phi) (x,y) = y $ for y in the boundary of any of
the discs $D^{o,2}(p_i)$. The second part of the lemma follows.

\vskip .2in

To finish the third part of the lemma, it must be checked that the
natural map $\rho: Top(M)/Top(M,k)  \to\  F(M,k)$ is a
homeomorphism.  Notice that part 2 of the lemma gives local
sections $\phi: U  \to\ Top(M) $. Thus consider the composite
$\lambda: U  \to\ Top(M)/Top(M,k)$ given by the composite $p\phi$
where $ p: G \to\ G/H$ is the  natural quotient map. Notice that
$\lambda: U \to\  \lambda(U)$ is a continuous bijection, and
$\lambda(U) = \phi^{-1}(U)$. Thus $\lambda(U)$ is open, and the
map $\phi$ is open.  Thus $\rho$ is open, and hence a
homeomorphism. The lemma follows.

\end{proof}

\vskip .2in

Next, the proof of Theorem~\ref{thm: Topfibrations} is given.

\begin{proof}

By Lemmas~\ref{lem: diskmap}, and \ref{lem: closedsubgroup}, 
$Top(M,k)$ is a closed subgroup of
$Top(M)$, and local sections exist for $Top(M)  \to\
Top(M)/Top(M,k)$ . Thus there is a principal fibration $Top(M,k)
\to\ Top(M)  \to\ Top(M)/Top(M,k)$, and the first part of the
theorem follows.

\vskip .2in

Furthermore, the natural evaluation map $Top(M)  \to\ F(M,k)$
factors through the quotient map $Top(M) \to\ Top(M)/Top(M,k)$.
There is the induced map $\rho:Top(M)/Top(M,k) \to\  F(M,k)$
which, by lemma 1.4 , is a homeomorphism, and part 2 of the
theorem follows.

\vskip .2in

The third statement in the theorem follows from a construction in
N.  Steenrod's book \cite {S}, at the foot of page 30. Namely, consider
$ H \to\ G  \to\  G/H $ where $H$ is a
closed subgroup of $G$, and the map $G  \to\  G/H $ admits "local
cross-sections", then $BH \to\ BG$ is the projection map in a
fibre bundle with fibre given by the space of left cosets $G/H$.
\end{proof}

As a consequence of Theorem~\ref{thm: Topfibrations}, 
we can describe some of the Borel constructions 
we have looked at in previous sections as $K(\pi,1)$ spaces 
where the group $\pi$ is given by certain mapping class groups.

\vskip .2in

\begin{thm}

\begin{enumerate}
\item If $ q \geq 3$, the spaces $ESO(3)\times_{SO(3)} F(S^2,q)$,
and $ESO(3)\times_{SO(3)} F(S^2,q)/{\Sigma_q}$ are respectively
$K( P\Gamma_0^q,1 )$, and  $K( \Gamma_0^q,1 )$.

\item If M = $S^1 \times S^1$, and  $q \geq 2 $, the spaces
$ETop(M) \times_{Top(M)} F(M,q)$, and $ETop(M) \times_{Top(M)}
F(M,q)/{\Sigma_q}$ are respectively $K( P\Gamma_1^q,1 )$, and  $K(
\Gamma_1^q,1 )$. Furthermore, in these cases $ETop(M)
\times_{Top(M)} F(M,q)$ is homotopy equivalent to $ESL(2,Z)
\times_{SL(2,Z)} F(S^1 \times S^1 -\{(1,1)\},q-1)$ where $SL(2,Z)$
acts on $S^1 \times S^1 -\{(1,1)\}$  by the formula

\begin{equation}
\begin{pmatrix}a&b\\
c&d\end{pmatrix}\begin{pmatrix}u\\v\end{pmatrix} = (u^av^b,u^cv^d)
\end{equation}

for

\begin{equation}
\begin{pmatrix}u \\ v\end{pmatrix}
\end{equation}

in $S^1 \times S^1$.

\item  If M is a surface of genus g without boundary ( possibly with punctures ),
and $ g \geq 2 $, the spaces $ETop(M) \times_{Top(M)} F(M,q)$ and
$ETop(M) \times_{Top(M)} F(M,q)/{\Sigma_q}$ are respectively $K(
P\Gamma_g^q,1 )$, and  $K( \Gamma_g^q,1 )$.

\item Let M be a closed orientable surface of genus g with $x$ a
point of M, and N = $M-\{x\}$. The spaces $ETop(N) \times_{Top(N)}
F(N,q)$, and $ETop(N) \times_{Top(N)} F(N,q)/{\Sigma_q}$ are
respectively $K( P\Gamma_g^{q,1},1 )$, and $K( \Gamma_g^{q,1},1)$.
\end{enumerate}
\end{thm}
\begin{proof}

Notice that Lemma 2 gives that the fundamental group of $ETop(M)
\times_{Top(M)} F(M,q)$ is isomorphic to $\pi_0(Top(M,k)$ =
$P\Gamma_g^k$. Similarly, the fundamental group of $ETop(M)
\times_{Top(M)} F(M,q)/{\Sigma_q}$ is $\Gamma_g^k$ The main task
is now to determine when $ETop(M) \times_{Top(M)} F(M,q)$  is a
$K(\pi,1)$. If the genus of M is at least 2, then the resulting
space is a $K(\pi,1)$, however, if the genus is 0, or 1, there are
some small modifications described below.

\vskip .2in

A theorem of Smale \cite {Smale} gives that the natural map
$SO(3) \to\ Diff^+(S^2)$ is a homotopy equivalence. Thus the
natural maps $SO(3) \to\ Diff^+(S^2) \to Top(S^2) $ as well as the
maps induced on the level of the Borel constructions
$ESO(3)\times_{SO(3)} F(S^2,q) \to ESO(3)\times_{Top(S^2)}
F(S^2,q)$ are homotopy equivalences. Since $ESO(3)\times_{SO(3)}
F(S^2,q)$ is a $K(\pi,1$ for $ q \geq 3 $ ( See section ?? in
these notes . ), the result for part (1) follows from the above
lemma.

\vskip .2in

To prove part(2) of the theorem, notice that a result of Earle,
and Eells \cite {EE}, notice that there is an
exact sequence of groups $ 1 \to\ Diff_0(S^1\times S^1) \to\
Diff^+(S^1\times S^1) \to\ SL(2,Z) \to\ 1$ where $
Diff_0(S^1\times S^1) $ is homotopy equivalent to $S^1 \times S^1
$ and so $ BDiff_0(S^1\times S^1) $ is homotopy equivalent to $
CP^{\infty} \times CP^{\infty} $. Furthermore, the map $S^1 \times
S^1 \to\ Diff_0(S^1\times S^1)$ given by rotations is a homotopy
equivalence by \cite {EE}. Thus the Borel construction $EG \times_G S^1\times S^1$ is
$K(SL(2,Z),1)$ where G = $Diff^+(S^1\times S^1)$.

\vskip .2in

Consider the Borel construction $EG \times_G F(S^1 \times S^1,q)$
together with the projection to $EG \times_G F(S^1 \times S^1,1)$
having fibre $ F(S^1 \times S^1 - Q_1,q-1)$. Recall that $EG
\times_G F(S^1 \times S^1,1)$ is $K(SL(2,Z),1)$, and observe that
the action of $SL(2,Z)$ on $ F(S^1 \times S^1 - \{(1,1)\},q-1)$ is
the diagonal natural action of $SL(2,Z)$ on $S^1 \times S^1 -
\{(1,1)\}$ where $\{(1,1)\}$ is the identity element in $S^1
\times S^1$ with the following action:

\begin{equation}
\begin{pmatrix}a&b\\
c&d\end{pmatrix}\begin{pmatrix}u\\v\end{pmatrix} =
(u^av^b,u^cv^d).
\end{equation}

\vskip .2in

Thus, if $q \geq 2 $, the fibre and base of the projection $EG
\times_G F(S^1 \times S^1,q) \to\ EG \times_G F(S^1 \times S^1,1)$
are $K(\pi,1)'s$. Thus $EG \times_G F(S^1 \times S^1,q)$ is also a
$K(\pi,1)$ if $ q \geq 2 $, and part (3) follows.

\vskip .2in

To prove part (3), notice that by \cite {EE},each path component of $ Diff^+(M) $ is
contractible if M is of genus $ g \geq 2 $. Thus both $BTop(M)$,
and $F(M,k)$ are Eilenberg-Mac Lane spaces of type $K(\pi, 1)$.

\vskip .2in

Part(4) is analogous as $BTop(M)$ is a $K(\pi,1)$ where N =
$M-\{x\}$. The theorem follows.

\end{proof}

\vskip .2in

Recall the definition of the pointed mapping class group:
The pointed mapping class group $\Gamma_g^{k,1}$ is the group of
path components of the orientation preserving homeomorphisms which
(1) preserve the point p in the surface, and (2) leave a set of k other distinct
points in M invariant, $\pi_0(Top(M,\{p\},Q_k)$. The pure pointed
mapping class group $P\Gamma_g^k$ is the kernel of the natural
homomorphism $\Gamma_g^{k,1}  \to\ \Sigma_k$.

\vskip .2in

\begin{cor}

\begin{enumerate}

\item If $q \geq 1$, then the fundamental group of

$ETop(M) \times_{Top(M)} {F(M,q+1)}/ \{1\ \times \Sigma_q \}$

is isomorphic to $\Gamma_g^{q,1}$.

\item If M is of genus zero, and $q \geq 2$, then
$ETop(M) \times_{Top(M)} {F(M,q+1)}/\{1\ \times \Sigma_q\}$
is a $K(\Gamma_0^{q,1},1)$.

\item If M is of genus one, and $q \geq 1$, then
$ETop(M) \times_{Top(M)} {F(M,q+1)}/\{1\ \times \Sigma_q\}$
is a $K(\Gamma_g^{q,1},1)$.

Furthermore if $q \geq 1$, then $ETop(M)\times_{Top(M)} F(M,q)$ is
homotopy equivalent to $ESL(2,Z)
\times_{SL(2,Z)} F(S^1 \times S^1 -\{(1,1)\},q-1)$ where $SL(2,Z)$
acts on $S^1 \times S^1 -\{(1,1)\}$  by the formula

\item If M is of genus at least two, and $q \geq 1$, then
$ETop(M) \times_{Top(M)} {F(M,q+1)}/\{1\ \times \Sigma_q\}$
is a $K(\Gamma_g^{q,1},1)$.

\end{enumerate}

\end{cor}

\begin{proof}
Consider $ETop(M) \times_{Top(M)} {F(M,q+1)}/{\{1\ \times \Sigma_q}$
together with the natural projection to $ETop(M) \times_{Top(M)}M$
with fibre $F(M - Q_1,q)/\Sigma_q$. Thus the space
$ETop(M) \times_{Top(M)}$ is $BTop(M,1)$ by
Theorem 1.2 . Furthermore, the $Top(M,1)$-action on $F(M - Q_1,q)/\Sigma_q$
is induced by the natural diagonal action on $M - Q_1$. Hence the
fundamental group is $\Gamma_g^{q,1}$. The hypotheses on the
number of points $q$ gives that the resulting spaces are
$K(\pi,1)'s$ by Theorem 1.2.

\vskip .2in

To finish the proof of the corollary, it suffices to notice that
$ETop(M)\times_{Top(M)} F(M,q)$ is
homotopy equivalent to $ESL(2,Z)
\times_{SL(2,Z)} F(S^1 \times S^1 -\{(1,1)\},q-1)$ where $SL(2,Z)$
acts on $S^1 \times S^1 -\{(1,1)\}$ where M is of genus 1 by the above remarks.

\end{proof}

\section{The Thom construction}

\subsection{Basic definitions}
In this section, we will be working in the category of pointed
spaces and maps. Recall the basic definitions:

\begin{defn}
A pointed space is a space $X$ together with a fixed basepoint $* \in X$.
A map $f: X \rightarrow Y$ between two pointed spaces is said to be
a pointed map if $f(*)=*$. A homotopy $F$ between two pointed maps
$f,g: X \rightarrow Y$ is said to preserve basepoints if 
$F(* \times I)=*$ and we write $f \simeq_* g$. If $f$ and $g$
are homotopic but not necessarily by a homotopy that preserves
basepoints, we write $f \simeq g$ and say $f$ is freely homotopic
to $g$. 
\end{defn}

We recall some basic constructions. See \cite{Bred} for 
proofs of the basic properties of these constructions.

\begin{defn}
Given two pointed spaces $X$ and $Y$, $X \times Y$ can be made
a pointed space by taking $(*,*)$ as basepoint. Define
$$
X \vee Y = \{ (x,y) \in X \times Y \text{ such that either }
x=* \text{ or } y=* \}.
$$
$X \vee Y$ is called the wedge product of $X$ and $Y$ and is a pointed space
consisting of the spaces $X=X \times *$ and $Y=* \times Y$ 
attached at the basepoint
$(*,*)$.
\end{defn}

\begin{defn}
Given two pointed spaces $X$ and $Y$,
$$
X \wedge Y = X \times Y / X \vee Y
$$
is called the smash product of $X$ and $Y$ and is 
itself a pointed space where we set the equivalence
class of $X \vee Y$ as the basepoint.
\end{defn}

The following lemma collects some useful facts about 
the wedge and smash products. The proof of these
facts is left to the reader and can be found in introductory
texts like \cite{Bred} or 
\cite{Span}.

\begin{lem}
\label{lem: basicwedgesmash}
Let $Top_*$ denote the category of pointed spaces and maps.
We will write $X=_*Y$ if $X$ and $Y$ are isomorphic in this
category. Then \\
\noindent
(a) $-\vee-$ and $-\wedge-$ are functors from $Top_*$ to itself, which
are covariant in both entries. If $f: X_1 \rightarrow X_2$ and
$g: Y_1 \rightarrow Y_2$ are two pointed maps we will denote the
maps given by these functors as
$f \vee g: X_1 \vee Y_1 \rightarrow X_2 \vee Y_2$ and
$f \wedge g: X_1 \wedge Y_1 \rightarrow X_2 \wedge Y_2$.
\\
\noindent
(b) $-\vee-$ equips the isomorphism classes of objects in $Top_*$ 
with the structure of a
commutative monoid with identity where the identity is the point space $*$.
In other words: \\
(i) $X \vee Y =_* Y \vee X$ \\
(ii) $(X \vee Y) \vee Z =_* X \vee (Y \vee Z)$ \\
(iii) $X \vee * =_* X$. \\
\noindent
(c) $-\wedge-$ equips the isomorphism classes of objects in
$Top_*$ with a commutative multiplication structure, 
which together with $\vee$
gives a commutative semiring structure on $Top_*$. (A semiring is something 
which satisfies
all the axioms of a ring except the existance of additive inverses.)
In other words: \\
(i) $X \wedge Y =_* Y \wedge X$ \\
(ii) $(X \wedge Y) \wedge Z =_* X \wedge (Y \wedge Z)$ \\ 
(iii) $X \wedge S^0 =_* X$ where $S^0=\{0,1\}$ with 
$0$ as the basepoint. \\ 
(iv) $(X \vee Y) \wedge Z = (X \wedge Z) \vee (Y \wedge Z)$. \\ 
\noindent
(d) $\vee$ and $\wedge$ are homotopy functors, i.e., if $f \simeq_*
f'$ and $g \simeq_* g'$ then we have $f \vee g \simeq_*
f' \vee g'$ and $f \wedge g \simeq_* f' \wedge g'$.
\end{lem}

\begin{rem}
The reader should be careful in interpreting the semiring structure in 
lemma~\ref{lem: basicwedgesmash} since the objects of $Top_*$ do not
form a set. However, in practice, this is not a problem as we
can always confine ourselves to a suitable set of objects if we wish to use
the semiring structure. The reader can check for example, that the
``subsemiring'' generated by $S^0$ is naturally identified with
the semiring of natural numbers. In fact, if we confine ourselves
to objects in $Top_*$ where the Euler characteristic $\chi$ is defined, 
and where we can identify $\tilde{H}_*(X \wedge Y)$ as $H_*(X \times Y, X 
\vee Y)$, 
then the map taking $X$ to $\chi(X)-1$ is a semiring morphism to the
integers.
\end{rem}

\begin{defn}
\label{rem: sigmakactiononwedge}
We use $X^{(k)}$ to stand for the $kth$ power of the pointed space $X$
under the smash product multiplication. It is understood then that
$X^{(0)}=S^0$ and $X^{(1)}=X$. Notice that, for $k \geq 1$, we have 
$$
X^{(k)} = X^k/F
$$
where $F=\{ (x_1,\dots,x_k) \in X^k \text{ such that } x_i = *
\text{ for some } i \}$, is the so called fattened wedge. 
Notice that the symmetric group $\Sigma_k$ acts on $X^k$
by permuting coordinates and that this action preserves $F$.
Thus this induces a natural action of $\Sigma_k$ on $X^{(k)}$
which fixes the basepoint.
\end{defn}

\begin{defn}
Given a space $X$, the unreduced suspension of $X$,
denoted $SX$, is the quotient space obtained from $X \times I$
by identifying $X \times \{0\}$ to a point and
$X \times \{1\}$ to another point. \\
Given a pointed space $X$, the reduced suspension of $X$,
is defined as
$$
\Sigma X = X \times I/(X \times \{0\} \cup * \times I \cup
X \times \{1\}).
$$
and is given the equivalence class of $* \times I$ as the basepoint.
\end{defn}

\begin{rem}
We will identify the 1-sphere $S^1$ as the quotient space formed
from $[0,1]$ where we identify $0$ and $1$ to form a common basepoint.
With this, it is easy to see that
$$
S^1 \wedge X =_* \Sigma X.
$$
Thus $(S^1)^{(k)}=S^k$ for all $k \geq 0$.
\end{rem}

\begin{defn}
A well pointed space $X$ is a pointed space where the inclusion
$* \rightarrow X$ is a cofibration. Recall that for such a space,
the reduced and unreduced suspensions are homotopy equivalent. (See
\cite{Bred} or \cite{Span}). Any $CW$-complex $X$ is well pointed
if we take the basepoint to be an element of the zero skeleton
$X_{(0)}$.
\end{defn}

\subsection{The Thom construction}

We now introduce an important construction.
Let $M$ be a free right $\Sigma_k$-space and $X$ be a well pointed space.
Then from the final remark in definition~\ref{rem: sigmakactiononwedge},
there is a natural $\Sigma_k$ action on $X^{(k)}$ which fixes the basepoint.
Thus we can form
$$
M \times_{\Sigma_k} X^{(k)} = M \times X^{(k)}/\sim,
$$ 
where $(m \sigma, \bar{x}) \sim (m, \sigma \bar{x})$ for all
$\sigma \in \Sigma_k$, $\bar{x} \in X^{(k)}$ and $m \in M$.

As in the Borel construction, one can easily show that as
the $\Sigma_k$ action on $M$ is free, the map 
$$
\pi: M \times_{\Sigma_k} X^{(k)} \rightarrow M/\Sigma_k,
$$
obtained by projecting on the first factor, is a fiber bundle with
fiber $X^{(k)}$.
Furthermore $\pi$ has a section $\sigma: M/\Sigma_k \rightarrow
M \times_{\Sigma_k} X^{(k)}$ defined by 
$\sigma(\bar{m})=(m,*)$ where $*$ is the basepoint of $X^{(k)}$.
(This section is well defined as $*$ is fixed under the $\Sigma_k$
action on $X^{(k)}$.)

Given a field $\FF$, we will now set out to find $H^*(\MsigX; \FF)$. 
To do that using the spectral sequence for the fiber bundle described 
above, we see that we first need to describe $H^*(X^{(k)}; \FF)$.
So let us do that now.

Recall that if the inclusion of $A$ into $X$ is a cofibration, then 
we have a natural isomorphism 
$H^*(X,A; \FF) \cong \bar{H}^*(X/A; \FF)$ where the bar signifies reduced 
cohomology. (See for example \cite{Bred}).

Now if $X$ and $Y$ are well pointed spaces, 
then the inclusion of $X \vee Y$ into $X \times Y$ is a cofibration?

Thus $H^*(X \times Y, X \vee Y; \FF) \cong \bar{H}^*(X \wedge Y ; \FF)$.
On the other hand, it is easy to argue that the long exact sequence in 
reduced cohomology for the pair $(X \times Y, X \vee Y)$ degenerates 
into split short exact sequences:

$$
0 \rightarrow H^*(X \times Y, X \vee Y; \FF) \rightarrow 
\bar{H}^*(X \times Y; \FF) \overset{i^*}{\rightarrow} 
\bar{H}^*(X \vee Y; \FF) \rightarrow 0
$$
 
Thus $\bar{H}^*(X \wedge Y; \FF)$ is algebra isomorphic to 
the kernel of $i^*$. We are able to get a complete description of this 
kernel via the K$\ddot{u}$nneth Theorem in the case when one of the spaces has 
finite dimensional $\FF$-cohomology in each dimension. (We need this 
condition, to apply the cohomology version of K$\ddot{u}$nneth's theorem - see 
\cite{Bred}).

Using this, one easily describes the cohomology algebra of the smash product 
in terms of the cohomology algebras of the original spaces. We 
state the result as a lemma and leave the completion of the details of its 
proof to the reader.

\begin{lem}
\label{lem: wedge}
If $X, Y$ are path connected, well pointed spaces and at least one of them 
has finite dimensional $\FF$-cohomology in each dimension, then 
$\bar{H}^*(X \wedge Y; \FF) \cong \bar{H}^*(X; \FF) \otimes \bar{H}^*(Y; \FF)$
as $\FF$-algebras (without identity element). 
\end{lem} 

It follows from lemma~\ref{lem: wedge} that for a path connected, well pointed 
space $X$, one has $\bar{H}^*(X^{(k)};\FF)$ is isomorphic to 
the tensor product of $k$ copies of the algebra $\bar{H}^*(X;\FF)$ with 
itself. 

\begin{defn}
Given a graded $\FF$-algebra $\mathfrak{A}$ with $\mathfrak{A}^0=0$ 
(so $\mathfrak{A}$ does not have an identity element), we define 
$T_k(\mathfrak{A})$ to be the $\FF$-algebra with identity obtained in 
the following way. We first take the tensor product of $k$-copies of 
$\mathfrak{A}$. This has nothing in degree zero. Finally, we include 
a 1-dimensional vector space in degree zero generated by an identity element.
\end{defn}

Thus we can restate our results as 
$H^*(X^{(k)}; \FF) \cong T_k(\bar{H}^*(X; \FF))$.

Fix a label space $X$. 
Now as we mentioned before, there is a natural left $\Sigma_k$ action on 
$X^{(k)}$ induced from the left $\Sigma_k$ action on $X^k$ given by 
$\sigma \cdot (x_1, \dots, x_k)=(x_{\sigma(1)}, \dots, x_{\sigma(k)})$.

Since cohomology is a contravariant functor, it is easy to check that we 
get a right $\Sigma_k$ action on $H^*(X^k; \FF)$. Let us describe this 
action using the K$\ddot{u}$nneth theorem to identify 
$H^*(X^k; \FF)$ as the tensor product of $k$ copies of $H^*(X; \FF)$.
It is easy to check that an element of the form $1 \otimes \dots \alpha 
\dots \otimes 1$ where $\alpha$ is in the $i$th spot gets taken 
under $\sigma \in \Sigma_k$  
to a similar element where $\alpha$ is in the $\sigma^{-1}(i)$th spot.
(To check this note that such elements correspond nicely to elements 
in the cohomology of the $k$-fold wedge product of $X$ where the statement 
is clear.)

This describes the right $\Sigma_k$ action on $H^*(X^k; \FF)$ completely 
as elements of the form $1 \otimes \dots \alpha \dots \otimes 1$ 
generate this cohomology as an algebra and $\Sigma_k$ acts via 
algebra maps. 

It is convenient, to switch this right $\Sigma_k$ action to a left 
$\Sigma_k$ action via a standard procedure. Given a right action of 
a group $G$ on some object $X$, one obtains a left action of $G$ on $X$ 
by defining
$$ g \cdot x \equiv x \cdot g^{-1} .$$ 

Doing this the left action of $\Sigma_k$ on $H^*(X^k; \FF)$ sends  
$1 \otimes \dots \alpha \dots \otimes 1$ where $\alpha$ is in the $i$th 
spot to the corresponding element where $\alpha$ is in the 
$\sigma(i)$th spot under $\sigma \in \Sigma_k$.

Thus we have in general that 
$$
\sigma \cdot (\alpha_1 \otimes \dots \otimes \alpha_k) 
= \pm \alpha_{\sigma(1)} \otimes \dots \otimes \alpha_{\sigma(k)}.
$$ 
The $\pm$ sign occurs because of the grading as the following 
example illustrates:

Let $k=2$ and $\sigma=(1,2)$, then 
\begin{align*}
\begin{split}
\sigma \cdot (\alpha \otimes \beta) &= \sigma \cdot (\alpha \otimes 1 \cup 
1 \otimes \beta) \\
&= (\sigma \cdot ( \alpha \otimes 1)) \cup (\sigma \cdot (1 \otimes \beta)) \\
&= (1 \otimes \alpha) \cup (\beta \otimes 1) \\
&= (-1)^{|\alpha||\beta|}(\beta \otimes 1 ) \cup (1 \otimes \alpha) \\
&= (-1)^{|\alpha||\beta|}(\beta \otimes \alpha) 
\end{split}
\end{align*}

Of course, this description of the left $\Sigma_k$ action on 
$H^*(X^k; \FF)$ restricts to a description of the left 
$\Sigma_k$ action on  
$H^*(X^{(k)}; \FF) \cong T_k(\bar{H}^*(X; \FF))$. 

One of the main cases we will be looking at is when $X$ is $S^d$, the 
$d$-sphere. In this case, we can make the sign in the left $\Sigma_k$ 
action on $H^*(X^{(k)}; \FF)$ explicit. 

\begin{defn}
A representation of $\Sigma_k$ on a $\FF$-vector space is said to 
be trivial if every element acts as the identity map of the vector 
space.

The sign representation of $\Sigma_k$ is a one dimensional 
$\FF$-vector space where the even elements of $\Sigma_k$ act as 
multiplication by $1$ while the odd elements of $\Sigma_k$ act 
as multiplication by $-1$. 

Notice that if the characteristic of $\FF$ is two, then the 
sign representation is actually trivial.
\end{defn}

First note, that the reduced 
cohomology of $S^d$ is zero except in degree $d$ where it is 
one dimensional generated by $\alpha$ say. 

Thus 
$\bar{H}^*((S^d)^{(k)}; \FF)$ is zero except in degree $kd$ 
where it is one dimensional generated by $T=\alpha \otimes \dots \otimes 
\alpha$. (In fact in this case we know $(S^d)^{(k)}=S^{dk}$ but we will not 
use that.)

Each element in $\Sigma_k$ takes $T$ to $ \pm T$. It is easy to 
see that if $d$ is even, $\sigma(T)=T$ for all $\sigma \in \Sigma_k$ 
while if $d$ is odd, $\sigma(T)=(-1)^{\epsilon(\sigma)}T$ 
where $\epsilon$ is the sign representation of $\Sigma_k$ into 
$\{-1,+1\} \subset \FF^*$. ($\FF^*$ stands for the group of nonzero 
elements of $\FF$.)

Thus we have shown the following useful proposition:

\begin{prop}
As a left $\Sigma_k$-module, 
$\bar{H}^*((S^d)^{(k)}; \FF)$ is concentrated in degree $kd$ and 
in that degree it is \\
(a) A trivial one dimensional $\Sigma_k$-module if $d$ is even. \\
(b) The one dimensional sign representation if $d$ is odd.  
\end{prop} 

Now let us calculate the cohomology of $\MsigX$ where $X=S^d$.
First recall that we had a fiber bundle 
$\pi: \MsigX \rightarrow M/\Sigma_k$ with fiber $X^{(k)}$. 
Thus we have a Serre spectral sequence with 
$$
E^{p,q}_2 = H^p(M/\Sigma_k; H^q(X^{(k)};\FF))
$$
abutting to $H^{p+q}(\MsigX; \FF)$.
The reader is warned that the coefficients in the $E_2$-term are twisted. 
In fact from the cover $M \rightarrow M/\Sigma_k$ we get a short exact 
sequence of groups 
$$
\pi_1(M) \rightarrow \pi_1(M/\Sigma_k) \overset{\lambda}{\rightarrow} 
\Sigma_k.
$$ 
The action of $\pi_1(M/\Sigma_k)$ on the cohomology of the fiber 
is easily checked to be given by 
the composition of $\lambda$ and the $\Sigma_k$-action on $H^*(X^{(k)}; \FF)$
described in the previous paragraphs.

Recall that the fiber bundle $\pi$ had a section $\sigma$. This means that in 
this spectral sequence, no differentials will hit the horizontal line 
$q=0$. Thus for general $X$, $E^{p,0}_2=E^{p,0}_{\infty}$ in this 
spectral sequence. 

Now for $X=S^d$, we have $E_2$ is concentrated on the horizontal lines 
$q=0$ and $q=kd$. Since we know no differential can hit the line $q=0$, 
we conclude all diffenrentials are zero in this spectral sequence 
and $E_2^{*,*}=E_{\infty}^{*,*}$. Thus we conclude for $X=S^d$, 
we have an isomorphism of vector spaces
$$
H^*(\MsigSd; \FF) \cong H^*(M/\Sigma_k; \FF) \oplus H^{*-kd}(M/\Sigma_k; 
H^{kd}((S^d)^{(k)}; \FF))
$$
where again recall that the last summand has a twisted coefficient.

The first summand above, is the part of $H^*(\MsigX; \FF)$ coming 
from the image of $\pi^*$ or in other words coming from the image of our 
section. It can be shown that the section above is a cofibration, so if 
we form $\MtsigX$ the space obtained from $\MsigX$ by collapsing the 
image of the section to a basepoint, we conclude that 

$$
\bar{H}^*(\MtsigSd; \FF) \cong H^{*-kd}(M/\Sigma_k; H^{kd}((S^d)^{(k)}; \FF))
$$
as algebras. (Here we notice that in the $E_{\infty}$ term above, 
once we collapse the image of the section to a basepoint, we lose 
the row $q=0$ and hence there are no lifting problems anymore since 
everything is concentrated in the row $q=kd$.)

This shows that the algebra structure of $\bar{H}^*(\MtsigSd; \FF)$ is 
trivial, i.e., the product of any two elements is zero.

\section{Lie Algebras and Kohno-Falk-Randell Theory}

The purpose of this section is to consider the functor from
groups to Lie algebras given by sending a group $G$ to the Lie algebra
obtained from the descending central series for $G$. 

The Lie
algebras associated in this way to some pure braid groups as well as the
fundamental groups of orbit configuration spaces appear in several
different mathematical contexts. 

In this section, a very useful
tool for the analysis of these Lie algebras, obtained by T. Kohno 
\cite{KO}, and
Falk, and Randell \cite{FR} is described. 
They used this tool to analyze the beautiful case where $G$ is the
$k$th pure Artin braid group. This general theory is described below
along with several examples related to braid groups and elliptic
curves.

We begin by defining some preliminary group theoretic concepts.

\begin{defn}
Given two subgroups $H, K$ of a group $G$, we define 
$[H,K]$ to be the subgroup of $G$ generated by commutators 
$[h,k]=h^{-1}k^{-1}hk$ where $h$ ranges over the elements of $H$ 
and $k$ ranges over the elements of $K$.

It is easy to check that $[H,K]$ is normal (characteristic) 
in $G$ if $H$ and $K$ 
are normal (characteristic) in $G$.
(Recall a subgroup of $G$ is called characteristic if it is invariant 
under every automorphism of $G$.)
\end{defn}

\begin{defn}
We define the descending central series of $G$, 
$$
G=\Gamma^1(G) \geq \Gamma^2(G)\geq ....\geq \Gamma^n(G) \geq \dots
$$ 
inductively by $G = \Gamma^1(G)$ and
$\Gamma^n(G) = [G, \Gamma^{n-1}(G)]$ for $n>1$.
\end{defn}

The following proposition collects some elementary facts about 
this series, the proofs of which are easy and left to the reader:
[references: see Magnus, Karass, Solitar, or Michael Vaughn-Lee]

\begin{prop}
\label{pro: assgrad} 
Let $\Gamma^n(G)$ denote the $n$th term in the descending central 
series, then: \\
(1) The group $\Gamma^{n}(G)$ is a normal (in fact characteristic) 
subgroup of $G$. \\
(2) The quotient group $E_0^n(G) = \Gamma^{n}(G)/ \Gamma^{n+1}(G)$ 
is abelian for all $n \geq 1$. \\
(3) The map of sets given by the commutator
$<,>: G \times G \to\ G$ by defining 
$$
<x,y> = x^{-1} y^{-1} xy
$$ 
induces a 
well-defined bilinear pairing
$$
[\cdot,\cdot]: E_0^n(G) \otimes_{\ZZ} E_0^m(G)  \to\ E_0^{n+m}(G)
$$
which is alternating, i.e., 
$$
[a,a] = 0$$ 
for all $a$,
and satisfies the 
Jacobi identity, i.e., 
$$
[[a,b],c] + [[b,c],a] + [[c,a],b]=0
$$ 
for all $a, b, c$. \\
(4) Thus $E_0^*(G)$ is a quasi-graded Lie algebra in the  
sense that if $a$ has degree $n$ and $b$ has degree $m$, then 
$[a,b]$ has degree $n+m$. The reader is warned that this notion 
of quasi-graded Lie algebra is different from the notion of a 
graded Lie algebra typically used in topology 
which is defined below for completeness. 
\end{prop}

\begin{defn}
A graded Lie algebra is a graded abelian group $E^*=\oplus_{i \in \ZZ} E^i$
together with a bilinear bracket $[\cdot,\cdot]: E^* \otimes E^* \to E^*$ 
satisfying: \\
$$
[a,b]=(-1)^{|a||b|}[b,a]
$$ 
and 
$$
(-1)^{|a||c|}[[a,b],c] + (-1)^{|b||a|}[[b,c],a] + (-1)^{|c||b|}[[c,a],b] 
= 0
$$
for all homogeneous $a,b,c \in E^*$. (Here $|a|$ denotes the degree of $a$ 
etc.)
\end{defn}

The motivation for the definition of a graded Lie algebra given above  
is that the homotopy groups $\{ \pi_n(X) | n \geq 2 \}$ 
of a well-pointed space $X$ fit together to give a
graded Lie algebra under the Whitehead product. 
(See \cite{Bred}.)

Consider the Lie algebra $E_0^*(G)$ obtained from the descending central
series for the group $G$ as described in proposition~\ref{pro: assgrad}. 
For each positive integer $q$,
there is a canonical graded Lie algebra
$E_0^*(G)_q$ obtained from $E_0^*(G)$, 
which will be useful in the next section,
and which is defined as follows.

\begin{defn}
Fix a positive integer $q$ and let $\Gamma^n(G)$ denote the $n$th
stage of the descending central series for $G$.
Define
$$
E_0^{i}(G)_q = \begin{cases} 
E_0^n(G) \text{ if } i=2nq \\
0 \text{ if }  i \neq 0 \text{ mod } 2q
\end{cases}
$$
Finally define the Lie bracket on $E_0^*(G)_q$ to be 
that induced from the one of $E_0^*(G).$
\end{defn}

We are now ready to look at the main theorem of this section.

\begin{thm}[T. Kohno, Falk-Randell]
\label{thm: Falkrandell}
Let
$$ 
1 \to\  A  \to\   B  \to\ C \to\ 1 
$$
be a split short exact sequence of groups such that
the conjugation action of C on $H_1(A)$ is trivial.

Then there is a short exact sequence of Lie algebras
$$ 
0 \to\  E_0^*(A)  \to\  E_0^*(B)  \to\  E_0^*(C)  \to\  0 
$$
which is split as a sequence of abelian groups. 
(Thus there is an isomorphism of abelian
groups  $E_0^n(A) \oplus E_0^n(C)  \cong E_0^n(B)$ but 
this isomorphism need not preserve the Lie algebra structure.)
\end{thm}
\begin{proof}

This proof follows that of Falk-Randell \cite{FR} and 
Xicot\'entcatl \cite{X}.
Consider the extension
$$ 
1 \to\ A  \overset{j}{\rightarrow}   B  \overset{p}{\rightarrow} C \to\ 1 
$$
and let $\sigma: C \to\ B$ be a splitting for $p$.
Observe that $p(b \sigma(p(b^{-1}))) = 1$ for all $b \in B$, so there
exists a unique element $a \in A$ with $j(a) = b(\sigma (p(b^{-1})))$ 
for all $b \in B$.
Thus, there is a well-defined function (which is not necessarily a
homomorphism) $\tau: B \to\ A $ defined by the formula
$ \tau(b) = j^{-1}(b\sigma (p(b^{-1})))$.

Notice that the trivial action of $C$ on $H_1(A)$ gives

\vskip 0.2in

\begin{enumerate}
\item $cac^{-1} = ax$ for a in A, and $x$ in $[A,A]$,
\item  $[B,A]$ is a subgroup of $[A,A]$, and
\item  $[\Gamma^n(B),\Gamma^m(A)]$ is a
subgroup of $\Gamma^{m+n}(A)$, and
\item  $\tau(\Gamma^n(B))$ is contained in $\Gamma^n(A)$.
\end{enumerate}

\vskip 0.2in

Since $\tau(\Gamma^{n}(B))$ is contained in $\Gamma^{n}(A)$,
for all n, there is a well-defined induced map of sets
$\tau: E_0^n(B) \to\ E_0^n(A)$ where $\tau$ is defined on an
equivalence class of b by the formula
$\tau([b])= \tau(b)$ ( as $\tau(bv)$ = $bv.\sigma p((bv)^{-1})$= 
$ \tau(b). \Gamma $ where $v$ and $\Gamma $ are in  $\Gamma^{n+1}(B)$.

Furthermore if b is in $ker(p)\cap \Gamma^n(B)$, then
$\tau(b) = b$. Thus $ \tau$ restricts to a function
$ \tau |_{ker(p)\cap \Gamma^n(B) }:ker(p)\cap \Gamma^n(B) \to\
\Gamma^n(A)$, and the homomorphism $ j: \Gamma^n(A) \to\ ker(p)\cap =
\Gamma^n(B)$ is
a group isomorphism. Thus, there is an exact sequence of groups
$ 1 \to\ \Gamma^n(A)  \to\   \Gamma^n(B)  \to\ \Gamma^n(C) \to\ 1$
which is split by the existence of $\sigma$.

Furthermore, if $[b]$ is in the kernel of
$E_0^n(p):E_0^n(B) \to\ E_0^n(C)$, then $E_0^n(\tau[b]) = [b]$.
Hence, there is a split short exact sequence
$ 0 \to\  E_0^n(A)  \to\  E_0^n(B)  \to\  E_0^n(C)  \to\  0 $.
The theorem follows from the above.

\end{proof}

The additive decomposition in theorem~\ref{thm: Falkrandell} 
may not necessarily preserve
the underlying Lie algebra structure. The Lie product is sometimes
"twisted", and quite interesting, as shall be seen in examples below. 
We will now look at some examples 
which demonstrate that both hypotheses on the theorem
are required (The existance of the 
splitting on the sequence of groups and the trivial action on homology.) 

\vskip 0.2in

{\bf Example 1:}
Let $F[S]$ denote the free group on a set $S$. 
Then $E_0^*(F[S])$ is isomorphic to the free Lie algebra
on $S$, denoted by $L[A_S]$. This is defined by letting 
$A_S$ be the free abelian group
with basis $S$, and then defining $L[A_S]$ to be the smallest 
sub-Lie algebra of the tensor algebra $T[A_S]$ containing 
$A_S$. [P. Hall, W. Magnus, J. P. Serre ].

\vskip 0.2in

{\bf Example 2:}
Let $G$ denote the $k$th pure Artin braid group $P_k$.
Fix a free abelian group $V_k$ with basis given by elements
$B_{i,j}$ for $1 \leq i < j \leq k$. Let $\mathcal{L}_k$
denote the quotient of the free Lie algebra $L[V_k]$ generated
by $V_k$, modulo the following ``infinitesimal braid relations'':

\begin{enumerate}
\item  $[B_{i,j}, B_{s,t}] =0 $ if $ \{ i,j\} \cap \{ s,t\} = \phi$,
\item  $[B_{i,j}, B_{i,t} + B_{t,j}] =0$ for $ 1 \leq j < t < i
\leq k$, and
\item  $[B_{t,j}, B_{i,j} + B_{i,t}] =0$ for $ 1 \leq j < t < i \leq k$.
\end{enumerate}

Then $E_0^*(P_k)$ is isomorphic to $ \mathcal{L}_k$. [T. Kohno,
Falk-Randell].

\vskip 0.2in

{\bf Example 3:}
Consider the orbit configuration space 
$F_G(M,k)$ where $M = \CC$, the complex numbers.
Let $G$ be the standard integral lattice 
$L=\ZZ + i\ZZ$, acting by translation on $\CC$.
Then $F_{L}(\CC,k)$ is a $K(\pi,1)$ which is studied in
[Cohen,Xicot\'encatl]. Let $F_{L}(\CC,k)$ be defined as above.
Picking a parametrized lattice $\ZZ + \omega \ZZ$ gives an analogous orbit
configuration space associated to an elliptic curve. These will be addressed
elsewhere.

\begin{enumerate}
\item The symmetric group $\Sigma_k$ acts on $F_{L}(\CC,k)$ and
the orbit space $F_{L}(\CC,k)/\Sigma_k  $ is homeomorphic to
the subspace of monic polynomials of degree $k$, $p(z) \in \CC[z]$, with
the property that the difference of any two roots of $p(z)$, 
$\alpha_i$, $\alpha_j$,
lies outside of the Gaussian integers.
\item It is the complement on $\CC^k $ of the infinite (affine) hyperplane
arrangement
$$ 
\mathcal{A} =
 \{ H^{\sigma}_{i,j}  \mid 1\leq j<i\leq k, \;\sigma\in L \}
$$
where
$ H^{\sigma}_{i,j} = \ker(z_i-z_j -\sigma) $
and $\sigma$ ranges over the lattice $L$.

\item It is an $L$-cover of the ordinary configuration
space of $k$ points in the torus $T = S^1 \times S^1$. This
is a special case of the results in Xicot\'encatl's Ph.D. thesis \cite{X} 
where it is proven that there exists a principal bundle
$$ L^k \to\  F_{L}(\CC,k) \to\  F(T, k).$$

\item The space $F_{L}(C,k)$ is a $K(\pi,1)$-space.
\item The fibration $F_{L}(C,k) \to\ F_{L}(C,k-1)$
has \\
(i) trivial local coefficients in homology, and \\
(ii) a cross-section.
\item Thus by theorem~\ref{thm: Falkrandell}, 
the Lie algebra attached to the descending central series of
$\pi_1(F_{L}(\CC,k))$ is additively isomorphic to the direct sum
$\bigoplus_{ 1 < i \leq k} L[i]$ where $L[i]$ is the free Lie algebra 
generated by elements
$B_{i,j}^{\sigma}$ for fixed i with $ 1 \leq j < i \leq k $, and 
$\sigma$ runs over the
elements of the lattice  $L$.
\item The relations are

$[B_{i,j}^{\sigma}, B_{k,i}^{\tau}] = [B_{k,i}^{\tau}, B_{k,j}^{\tau
+ \sigma}]$

$[B_{i,j}^{\sigma}, B_{k,j}^{\tau}] = [B_{k,j}^{\tau}, B_{k,i}^{\tau -
\sigma}]$
\end{enumerate}

\vskip 0.2in

The next list of examples gives short exact sequences of groups
such that the conclusion of the Kohno-Falk-Randell theorem fails, and where
one of the hypotheses in the theorem does not hold.

\vskip 0.2in 

{\bf Example 4:} 
Consider a split short exact sequence of groups 
$$ 
1 \to\  A  \to\   B  \to\ C \to\ 1 
$$
where both $A$ and $C$ are abelian. 

Thus if $n \geq 2$, both $E_0^n(A)$, and $E_0^n(C)$ are
trivial. However, it may well be the case that $B$ is not abelian, for 
example if $B$ is a nontrivial semi-direct product of $A$ and $C$, then 
$E_0^2(B)$ will be non-trivial and hence is not the sum of $E_0^2(A)$ 
with $E_0^2(C)$, thus spoiling
the conclusion of the Kohno-Falk-Randell theorem. 
Notice here that in this case $C$ will act nontrivially on $H_1(A)=A$.

The simplest example of this sort is 
given by taking the symmetric group on $3$ letters 
$\Sigma_3$ as $B$. If we take $A$ to be the normal Sylow-$3$ group 
of order $3$ and $C$ to be the group of order $2$, 
the group extension formed by $A, B$ and $C$ is split, but the
action of $C$ on the first homology group of $A$ is non-trivial. 
In fact $E_0^2(\Sigma_3)=\ZZ/3\ZZ$.

\vskip 0.2in

{\bf Example 5:}
Next consider group extensions

$ 1 \to\  A  \to\   B  \to\ C \to\ 1 $ where $A = \ZZ/2\ZZ$. Since
the only automorphism of $\ZZ/2\ZZ$ is the identity, the action of $C$ on 
$H_*(A)$ is always trivial. If this extension fails to
split, then the conclusion of the Kohno-Falk-Randell theorem
may be spoiled. For example we may take $A=C=\ZZ/2\ZZ$ and 
$B=\ZZ/4\ZZ$. Then $E_0^1(B) \neq E_0^1(A) \oplus E_0^1(C)$.

Another less trivial example is provided by non-abelian
extraspecial $2$-groups where  $C = (\ZZ/2\ZZ)^n $. In this case
$E_0^2(B)$ is again isomorphic to $\ZZ/2\ZZ$, but $E_0^2(A) \oplus
E_0^2(C)$ is the trivial group. Two more specific examples where
$B$ is non-abelian, and $C =  (\ZZ/2\ZZ)^2 $ are given by $D_8$, the dihedral
group of order $8$, and $Q_8$, the quaternion group of order $8$.

\vskip 0.2in

{\bf Example 6:} 
Assume that the group $C$ 
in the extension $ 1 \to\  A  \to\   B  \to\ C \to\ 1 $
is free. Then this extension is split. This setting gives infinite examples
where the conclusion of the Kohno-Falk-Randell theorem
may be spoiled. A specific example is
given by $C$ 
the free group generated by a finite set $S$ of cardinality $n$,
$B$ is the free group generated by the coproduct of the two non-empty sets
$ S \amalg T $ where $T$ has cardinality 1 with B = $ F[S \amalg T] $
and where the map $p: B  \to\ C$ given
by  the natural projection.

Then the above extension is split, but $E_0^1(A)$ is a countably infinitely
generated free abelian group, while $E_0^1(B)$ is a free abelian
group of rank n+1. Thus the natural map $E_0^1(A) \to\ E_0^1(B)$
has a kernel.

\vskip 0.2in

{\bf Example 7:}
If M is a punctured surface of genus greater than 0, and $ k \geq 2$
then the fibrations $ F(M,k) \to\ F(M,k-1) $  have sections, but
the local coefficient system is non-trivial ( as can be seen by
inpsection of the relevant Dehn twist ). Thus the Lie algebras
attached to the descending central series for the pure braid
groups of these surfaces is not clear. One case above addresses
this structure by considering the group that is the kernel of the
map

$ \pi_1F(M,k) \to\ \pi_1M^k $ that is induce by the inclusion map.

When M is a torus, the kernel of this last map satisfies the hypotheses
for the Kohno-Falk-Randell theorem. The resulting Lie algebra is given above 
[CX].

When M is any closed surface of genus at least one, it seems
likely that this kernel always satisfies the hypotheses of the
Falk-Randell-Kohno theorem.

\vskip 0.2in

Let $H$ denote the upper 1/2-plane, and let $\Gamma $ be a subgroup of
$SL(2,Z)$ that acts properly discontinuously on $H$. One might conjecture
that the fundamental groups of the orbit configuration spaces
$ \pi_1 F_{\Gamma}(H,k) $ satisfy the hypotheses of the
Kohno-Falk-Randell theorem, and thus the conclusion.

By \cite{X}, there is a short exact sequence of groups
$$ 
1 \to\  \pi_1 F_{\Gamma}(H,k)  \to\  \pi_1 F(M,k)  \to\ \pi_1(M^k) \to\ 1 
$$
where $H/\Gamma = M$. The point of this is that the orbit
configuration space has a ``nice'' associated Lie algebra. The above
extension tweezes apart two different phenomona in these Lie
algebras.

\newcommand{\colimn}
{\underset{\overset{\longrightarrow}{n}}{\hbox{ colim}}}
\newcommand{\colimg}{\underset{\overset
{\longrightarrow}{g}}{\hbox{ colim}}}
\newcommand{\lra}{\longrightarrow}

\section{Loop spaces of configuration spaces, and Lie algebras}

This section is about 2 possibly different constructions which are
in fact the same. The subject of this section is loop spaces of configuration
spaces and their relationship to the Lie algebra attached to the descending central
series of the pure braid group as described in the previous section.
The purpose of this section is to show that these Lie algebras, apart from
a formal degree shift, are given by the homotopy groups of configuration spaces
for points in complex n-space, $n > 1$, modulo torsion. This theorem, first proven
in \cite {FH1}, and subsequently in \cite {CG} is a
special case of a more general result , and which applies to
further analogues of pure braid groups. One example, the space of
monic polynomials where the differences of the roots lie ouside of
the Gaussian integers provides another example, as well as certain
choices of orbit configuration spaces are also described below.

\vskip .2in
Some additional discussion ( no proofs !) are given concerning
other related constructions which exhibit properties like braid
groups, and appear in the space of loops on a configruation space.
These loop spaces may be thought of as braids on a manifold, or as
trajectories of distinct particles moving through a time parameter
that start and quit in the same position.

\vskip .2in

These constructions also "fit" into several dfferent contexts. One
of which is that these spaces admit interpretations in terms of Vassiliev invariants
of braids, and knots. This subject will
not be addressed here; some information is given in
\cite {CG}. Indeed, one motivation for including this information
here is that loop spaces of configuration spaces keep track of
paths of distinct particles parametrized by time as they move through a manifold,
and are essentially braids on a manifold. There is a rich homological
structure attached to these paths, as well as a close connection to invariants of knots.

\vskip .2in

Recall the regrading of Lie algebras $E_0^*(G)_{m}$ for $ m > 0$
given in the previous section where $E_0^*(G)$ denotes the Lie
algebra attached to the descending central series of a discrete
group G.

\vskip .2in
\begin{thm}

If $m \geq 1$, then the homology of the loop space of the
configuration space $\Omega F(R^{2m+2},k)$ is isomorphic to the
universal enveloping algebra of the graded Lie algebra
$E_0^*(P_k)_{m}$. Furthermore,

\begin{enumerate}
\item the image of the classical
Hurewicz homomorphism $\pi_*(\Omega F(R^{2m+2},k))  \to\
H_*(\Omega F(R^{2m+2},k))$ is isomorphic to $E_0^*(P_k)_{m}$,
\item the Hurewicz homomorphism induces an isomorphism of graded
Lie algebras

$\pi_*(\Omega F(R^{2m+2},k))/torsion  \to\ PrimH_*(\Omega F(R^{2m+2},k))$

where Prim(.) denotes the module of
primitive elements, and the Lie algebra structure of the source is
given by the classical Samelson product.
\end{enumerate}
\end{thm}

Namely, the homotopy groups of the loop space of the configuration
space of k points in an even dimensional euclidean space $R^{2m+2}$,
modulo torsion, admits the structure of a graded Lie algebra
induced by the classical Samelson product. That Lie algebra is
isomorphic to $E_0^*(P_k)_m$ as described in the previous section.

\vskip .2in

Similar results apply to other analogues of pure braid groups.
Further work of \cite {X}, \cite {DC}, and the first
author show that an analogous theorem holds for some other groups
that are "close" to braid groups arising from some fibred
$K(\pi,1)$ hyperplane arrangements. Analogous results hold
for "orbit configuration spaces" for groups acting freely on the
upper 1/2-plane, and for some lattices acting on $C$ \cite {CX}.

\vskip .2in

Here, consider $F_G(M,k)$ in the case when $M = C $, the complex numbers,
and $G$ is the integral lattice $L$ = $Z + iZ$, acting by translation on $C$.
One of the consequences of the theorem below is that the Lie algebra obtained
from the fundamental group of the associated orbit configuration space also
gives the Lie algebra obtained form the higher homotopy groups of the
"higher dimensional analogues" of this arrangement. Two examples illustrating this
behavior are \cite {X}, \cite {DC}, and are described here as well as general theorem
about analogous spaces.

\vskip .2in

\begin{thm}

Let $F_{L}(C,k)$ be defined as above.

\begin{enumerate}
\item The symmetric group $\Sigma_k$ acts on $F_{L}(C,k)$ and
the orbit space $F_{L}(C,k) / \Sigma_k  $ is homeomorphic to
the subspace of monic polynomials of degree $k$, $p(z) \in C[z]$, with
the property that the difference of any two roots of $p(z)$, $\alpha_i$, $\alpha_j$,
lies outside of the Gaussian integers.
\item It is the complement in $C^k $ of the infinite (affine) hyperplane
arrangement
$$ \mathcal{A} =
 \{ H^{\sigma}_{i,j}  \mid 1\leq j<i\leq k, \;\sigma\in L \}$$
where
$  H^{\sigma}_{i,j} = \ker(z_i-z_j -\sigma) $.
\item It is an ${L}^k$-cover of the ordinary configuration
space of $k$ points in the torus $T = S^1 \times S^1$. This
is a special case of the results in \cite{X}, which gives
the existence of a principal bundle
$$ L^k \longrightarrow F_{L}(C,k) \longrightarrow F(T, k) .$$
\item The space $F_{L}(C,k)$ is a $K(\pi,1)$.
\item The fibration $F_{L}(C,k) \to\ F_{L}(C,k-1)$
has (i) trivial local coefficients in homology, and (ii) a cross-section.
\item Thus the Lie algebra given by the associated graded for
the descending central series of
$ \pi_1(F_{L}(C,k))$ is additively isomorphic to the direct sum
$\bigoplus_{ 1 < i \leq k} L[i]$ where $L[i]$ is the free Lie algebra generated by elements
$B_{i,j}^{\sigma}$ for fixed i with $ 1 \leq j < i \leq k $, and $\sigma$ runs over the
elements of the lattice $L$.

\item The relations are

$[B_{i,j}^{\sigma}, B_{k,i}^{\tau}] = [B_{k,i}^{\tau},
B_{k,j}^{\tau + \sigma}]$

$[B_{i,j}^{\sigma}, B_{k,j}^{\tau}] = [B_{k,j}^{\tau},
B_{k,i}^{\tau - \sigma}]$

\item The integral homology of $F_{L}(C, k)$ is additively given by
$$ H_* F_{L}(C, \ell) \cong  H_* (C_1) \otimes H_* (C_2) \otimes
\dots \otimes H_* (C_{k-1})
$$
where $C_i$ is the infinite bouquet of circles $\bigvee_{|Q_i^{L}|} S^1$
and $Q_i^{L}$ as defined in the beginning of next section.

\end{enumerate}
\end{thm}

\vskip .2in

Consider the "orbit configuration space" $F_{L}(C \times R^n,k)$
where $L$ operates diagonally on $C \times R^n$, and trivially on
$R^n$.

\vskip .2in
\begin{thm} Assume that $ q \geq 1 $.

\begin{enumerate}
\item The loop space $\Omega F_{L}(C \times R^{2q},k)$ is homotopy
equivalent to the product
$$ \prod_{1 \leq i \leq k-1} \Omega (C \times R^{2q} -Q_i^{L}) $$
( although this product decomposition is not multiplicative ).

\item The integral homology of $\Omega F_{L}( C \times R^{2q},k)$ is
isomorphic to

$$ \bigotimes_{1 \leq i \leq k-1} H_*(\Omega ( C \times R^{2q}-Q_i^{L}))$$
as a coalgebra.

\item The Lie algebra of primitives is isomorphic to the Lie
algebra given by  \\
$\pi_*(\Omega F_{L}(C \times R^{2q},k))/torsion$.

\item The Lie algebra of  of primitive elements in the homology of
$\Omega F_{L}( C \times R^{2q},k)$ is a direct sum of free (graded)
Lie algebras $\bigoplus_{ 1 < i \leq k} L[i]$ where $L[i]$ is the free
graded Lie algebra generated by elements
$B_{i,j}^{\sigma}$ of degree 2q for fixed i with $ 1 \leq j < i \leq k $,
and $\sigma$ runs over the elements of the lattice  ${L}$.
The relations are

$[B_{i,j}^{\sigma}, B_{k,i}^{\tau}] = [B_{k,i}^{\tau}, B_{k,j}^{\tau +
\sigma}]$

$[B_{i,j}^{\sigma}, B_{k,j}^{\tau}] = [B_{k,j}^{\tau}, B_{k,i}^{\tau -
\sigma}]$

\item The Lie algebras $\pi_*(\Omega F_{L}(C^q,k))$ modulo torsion,
and  $E_0^*(F_{L}(C,k))_q$ are isomorphic as Lie algebras.

\end{enumerate}

\end{thm}

\vskip .2in

Some of these Lie algebras occur for comparatively general reasons
as exemplified by the next theorem. The following general theorem
does not specifiy the structure constants for the underlying Lie
algebra, but shows that the Lie algebras addressed above
fit in a wider context.

\vskip .2in


\begin{thm} Assume that $n \geq 3 $.
Let $X(R^n,k) \to\ X(R^n,k-1)$ be a fibration which satisfies the
following properties:
\begin{enumerate}
\item The fibre of  $X( R^n,k) \to\ X( R^n,k-1)$ is $R^n - S_k$
where $S_k$ is a discrete subspace of $R^n$ of fixed (not neccessarily finite)
cardinality depending on $k$.
\item Each fibration $X(R^n,k) \to\ X(R^n,k-1)$ admits a cross-section.
\item The space $X(R^n,1)$ is $R^n $ with $n \geq 3 $.
\end{enumerate}
Then
\begin{enumerate}
\item  There is a homotopy equivalence
$\Omega X(R^n,k)  \to\  \prod_{1 \leq  i \leq k-1} \Omega(R^n - S_i)$.
\item The homology of $\Omega X(R^n,k)$ is torsion free, and is
isomorphic to \\
$ \bigotimes_{1 \leq  i \leq k-1} H_*(\Omega(R^n - S_i))$ as a
coalgebra.
\item The module of primitives in the integer homology
of  $\Omega X(R^n,k)$ is isomorphic to \\
$E_0^*(\pi_*(\Omega X(R^n, k)))$ modulo torsion as a Lie algebra.
\end{enumerate}
\end{thm}

\vskip .2in

In many examples which arise from fibre type arrangements, a similar conclusion
holds as that given in Theorem 1.1, and Theorem 1.3 part (5) above. Namely,
there are $K(G,1)'s$ in "good" cases where
$\pi_*(\Omega X(R^n, k))$ modulo torsion is isomorphic
to $E_0^*(G)_q$. The last section of this article contains
some speculation as to where, and how these structures fit.
It is a general theorem that if X is a 1-connected CW complex,
there is a functor $\Theta(X)$ = $K(G_X,1)$ where $G_X$ is a
filtered group such that the associated graded Lie algebra
tensored with the rational numbers gives the so-called
"homotopy Lie algebra" for the loop space of the rationalization
of X. \cite {CS}. This theorem admits some overlap with work of
T. Kohno, and T. Oda \cite {KO} on the descending central series
of the pure braid group of an algebraic curve.

\vskip .2in

\section{Proof of the theorem 1.4}

Recall that a multiplicative fibration with section is homotopy
equivalent to a product. Thus $\Omega X(R^n,k)$ is homotopy
equivalent to $\Omega X(R^n,k - 1)\times \Omega(R^n - S_{k-1})$, and
the first part of the theorem follows by induction.

\vskip .2in

The second part of the theorem follows from the {\it K\"unneth}
theorem, and part 1 of the theorem.

\vskip .2in

Since $R^n - S_{k-1}$ has the homotopy type of a (possibly
infinite ) bouquet of $(n-1)$-spheres, the homology of its loop
space follows from the Bott-Samelson theorem. In this case, it is
well-known that there are isomorphisms of Lie algebras
$\pi_*(\Omega(R^n - S_{k-1})/torsion  \to\ PrimH_*(\Omega(R^n - S_{k-1})$.
\vskip .2in

Furthermore, the existence
of sections implies that the Hurewicz homomorphism
$\pi_*(\Omega X(R^n,k) \to\ Prim H_*(\Omega X(R^n,k))$
is a surjection. Since this map is an injection, the theorem follows.

\vskip .2in

\section{Proof of Theorem 1.3}

\vskip .2in

Consider the fibration with section
$ F_{L}(C \times R^{2q},k) \to\   F_{L}(C \times R^{2q},k-1)$.
The fibre of this map is $C \times R^{2q} -Q_{k-1}^{L}$. By
Theorem 1, the conclusions of Theorem 3 all follow except possibly
the last two which state the precise extension of Lie algebras.

\vskip .2in
To finish, it suffices to prove parts (4), and (5) of the theorem
by a direct comparison of the two Lie algebras
$\pi_*(\Omega F_{L}(C \times R^{2q},k))/torsion$,
and  $E_0^*(F_{L}(C,k))_q$.

\vskip .2in

Thus define maps analogous to those in the proof of the relations for
the Lie algebra attached to the descending central series above:
$ F_i: S^{2q+1} \times  S^{2q+1}  \to\  F_{L}(C \times R^{2q},3)$
by the formula
\begin{align*}
F_1(z, w) & = (q_1, q_1 + \sigma + \frac{z}{8}, q_1+\tau+\frac{w}{16}) \\
F_2(z, w) & = (q_1, q_1+\sigma+\frac{z}{8}, q_1+\sigma+\tau+\frac{z}{8}+\frac{w}{16}).
\end{align*}

\vskip .2in

Consider the loopings of these maps
$ \Omega(F_i): \Omega( S^{2q + 1} \times  S^{2q + 1})  \to\  \Omega( F_{L}(C \times R^{2q},3)).$
Notice that the fundamental cycles in degree $2q$ for the
integer homology of $\Omega( S^{2q+1} \times  S^{2q+1})$ commute.
Thus it suffices to calculate the image of the fundamental cycles
in the homology of $\Omega( F_{L}(C \times R^{2q},3))$. This gives the
precise relations as stated in parts (4)-(5) of Theorem 3 which
follows at once.

\vskip .2in

\vskip .2in

Consider the Lie algebra obtained from the descending central
series for the group G. For each strictly positive integer q,
there is a canonical (and trivially defined) graded Lie algebra
$E_0^*(G)_q$ attached to the one obtained from the descending
central series for G, and which is defined as follows.

\vskip .2in

\begin{enumerate}

\item Fix a strictly positive integer q.
\item Let $\Gamma^n(G)$ denote the n-th
stage of the descending central series for G.
\item $E_0^{2nq}(G)_q = \Gamma^n(G) /  \Gamma^{n+1}(G),$
\item $E_0^{i}(G)_q = \{0\}$, if i is non-zero modulo $2q$,  and
\item the Lie bracket is induced by that for the associated
graded for the $\Gamma^n(G)$.

\end{enumerate}

The main theorem here is an interpretation of the results in [CG], and
[FH1] concerning the homology of the loop space of configuration spaces.

\vskip .2in

\begin{thm}

If $m \geq 1$, then the homology of the loop space of the
configuration space $\Omega F(R^{2m+2},k)$ is isomorphic to the
universal enveloping algebra of the graded Lie algebra
$E_0^*(P_k)_{2m}$. Furthermore, the following are satisfied.

\begin{enumerate}

\item  The image of the classical
Hurewicz homomorphism $\pi_*(\Omega F(R^{2m+2},k))  \to\
H_*(\Omega F(R^{2m+2},k))$ is isomorphic to $E_0^*(P_k)_{2m}$,
\item the Hurewicz homomorphism induces an isomorphism of graded
Lie algebras where where Prim(.) denotes the module of
primitive elements with the Lie algebra structure of the source
induced by the classical Samelson product:

$\pi_*(\Omega F(R^{2m+2},k))/torsion  \to\ PrimH_*(\Omega
F(R^{2m+2},k))$

\item If  $ q\geq 1 $, the Euler-Poincare' series
for the homology of $H_*( \Omega F(R^{q+2},n);\Bbb Z)$ is given as
follows:

$[(1-t^q)(1-2t^q)....(1-(n-1)t^q)]^{-1}$.

\end{enumerate}
\end{thm}

Namely, the homotopy groups of the loop space of the configuration
space of k points in an even dimensional euclidean space $R^{2m + 2}$,
modulo torsion, admits the structure of a graded Lie algebra
induced by the classical Samelson product. That Lie algebra is
isomorphic to $E_0^*(P_k)_m$, the Lie algebra that is "universal"
for the Yang-Baxter-Lie relations.

\vskip .2in

The theorem above appears to be related to beautiful results of T. Kohno and others
[H,K,K2,FR,W] who consider the relationship between Vassiliev invariants of braids
as well as other structures. Kohno has recently considered
the homology of the loop space for configurations in $R^3$.
In particular, the universal enveloping algebra of $E_0^*(P_k)$
regarded as a graded abelian group has Euler-Poincare' series given
by $[(1-t)(1-2t)....(1-(n-1)t)]^{-1}$ while the Euler-Poincare' series
for the homology of $H_*( \Omega F(R^{q+2},n);\Bbb Z)$ is
$[(1-t^q)(1-2t^q)....(1-(n-1)t^q)]^{-1}$.

\vskip .2in

Given a fibre type arrangement $X(k,R^n)$ with fibrations
$X(k,R^n) \to\ X(k-1,R^n)$ having sections with fibre given by
$R^n - S$ where S is a discrete set, consider the graded Lie algebra
$E_0^*(\pi_1(X(k,R^2)))_q$. A conjecture stated in [CX]
suggests that $E_0^*(\pi_1(X(k,R^2)))_q$ is isomorphic to the Lie
algebra of primitive elements in the homology of $\Omega
X(k,R^{2q+2})$ for many choices of $X(k,R^n)$.

\vskip .2in

There is more to this story. Interesting examples are given by the
pure braid groups for "orbit configuration spaces" in $C^*$. In
particular, the "orbit configuration space" $F_G(M,k)$ is the space
of ordered k tuples of points in M that lie on distinct orbits of a
free G action on M. Work of M. Xicot\'encatl [X], and D. Cohen [C] imply this
conjecture for the associated pure braid groups for
"orbit configuration spaces" in $C^*$.

\vskip .2in

Namely, this conjecture is correct for spaces $F_{Z/qZ}(C^n-\{0\},k) $
where $ Z/qZ $ is a finite cyclic group acting freely by
rotations on $ C^n-\{0\}$. If $ n > 1 $, the Lie algebra obtained from
the homotopy groups modulo torsion of the loop spaces for these choices of
"orbit configuration spaces" is isomorphic to the Lie algebra obtained
from the descending central series for $\pi_1F_{Z/qZ}(C^n-\{0\})$
[C],[X].

\vskip .2in

Further work of M. Xicot\'encatl, D. Cohen, and the first
author shows that an analogous theorem holds for some other groups
that are "close" to braid groups, and arise from some fibred
$K(\pi,1)$ hyperplane arrangements. Some similar results hold
for "orbit configuration spaces" for groups acting freely on the
upper 1/2-plane, and for some lattices acting on $C$ [CX].

\vskip .2in

There are related groups that share some common properties here.
Let $ Hom^{coalg}( T[v],H)$ denote the set of coalgebra
morphisms with source given by the tensor algebra over the
integers with a single primitive algebra generator v in degree 1.
Furthermore, the target H is a Hopf algebra with conjugation
(antipode). Recall that this set is naturally a group with
multiplication induced by the coproduct for the source and product
for the target with inverses induced by the conjugation in H [MM].

\vskip .2in

There are groups $ Hom^{coalg}(T[v],H_*\Omega F(M,k))$ where
M is a manifold for which the homology of its' loop space is
torsion free. When M is euclidean space, this group is not
isomorphic to the pure braid group, but in this case the associated Lie
algebra resembles $E_0^*(P_k)_q$. The first approximation in this
direction is as follows.

\vskip .2in

\begin{thm}{[S]}

If H is isomorphic to a tensor algebra generated by a rational
vector space of dimension q concentrated in a fixed even degree
that is strictly greater than 0, S is a set of cardinality q,
and $F[S]_M $ is the Mal\^cev completion of F[S],
then there is an isomorphism of groups
$$F[S]_M \to\ Hom^{coalg}(T[v], H).$$

\end{thm}

A similar calculation gives the following.

\vskip .2in

\begin{thm} {[CS]}
If $m \geq 1 $, the group $ Hom^{coalg}( T[v],H_*\Omega
F(R^{2m + 2 },k))$ is filtered such that the associated graded is
a graded Lie algebra which when tensored with $ \Bbb Q $ is isomorphic
to $E_0^*(P_k)_{2m} \otimes \Bbb Q $
\end{thm}

The group $ Hom^{coalg}( T[v],H_*\Omega(X))$ accepts
homomorphisms from the group of homotopy classes of pointed maps
$[\Omega S^2, \Omega(X)]$. The point is that this construction
provides a group theoretic analogue of the classical Hurewicz
homomorphism, and that these groups are more primitive versions of
homotopy groups. The actual Hurewicz map is the induced map on the
level of associated graded groups. In addition, one obtains
further braid-like groups by replacing euclidean space by other
manifolds M.

\section{Proof of Theorem 1.1}

The main theorem is essentially proven in \cite {CG}, and \cite {FH1}
except for the statement about the module of primitives. In that article, there
are maps constructed

$ B_{i,j} : S^{n-2}  \to\  \Omega F(R^n,k)$

for $ k \geq i > j \geq 1$ such that the image of the fundamental
cycles in the homology of $\Omega F(R^n,k)$ are (1) non-zero, and
(2) the homology of $\Omega F(R^n,k)$ is generated by these
classes as a Hopf algebra. Furthermore, Samelson products of the
generators map to the analogous Lie element in the homology of
$\Omega F(R^n,k)$ by the above construction.

\vskip .2in
These relations arise as follows.
Define maps

$$ \gamma_i : S^{n-1} \times S^{n-1} \to F(\RR^n, 3)$$

by the following formulas where $\| u \| = \| v \| =1$.

\begin{description}
\item[(i)] $\gamma_1(u,v) = (0,u,2v)$, and
\item[(ii)] $\gamma_2(u,v) = (0,2u,v)$.
\end{description}

\vskip .2in

Next, recall that the class $A_{i,j}$ is defined by the equation
$$A_{i,j} = \pi^*_{i,j} (\iota)$$ where $\pi_{i,j} : F(\RR^n, k)
\to F(\RR^n,2)$ denotes projection on the $(i,j)$ coordinates and
$\iota$ is a fixed fundamental cycle for $S^{n-1}$ \cite{CS, CG}.

\begin{lem} If $ n \geq 2$, then
\begin{enumerate}
\item $\gamma_1^*A_{2,1} = \iota \otimes 1,$
\item $\gamma_1^*A_{3,1} = 1 \otimes \iota,$
\item $\gamma_1^*A_{3,2} = 1 \otimes \iota,$
\item $\gamma_2^*A_{2,1} = \iota \otimes 1,$
\item $\gamma_2^*A_{3,1} = 1 \otimes \iota, and$
\item $\gamma_2^*A_{3,2} = \iota \otimes 1.$
\end{enumerate}

\end{lem}

\vskip .2in
\begin{proof}
Since all the cases above are similar, one will be worked out.
Notice that  $\gamma_1$ composed with the projection $\pi_{1,2}$
is homotopic to first coordinate projection from the product
$S^{n-1} \times S^{n-1}$ to $S^{n-1}$. This suffices.
\end{proof}
\vskip .2in
Direct dualization gives the following lemma with the details of
proof omitted.

\begin{lem} If $ n \geq 2$, then

\begin{enumerate}
\item  $[B_{i,j}, B_{s,t}] =0 $ if $ \{ i,j\} \cap \{ s,t\} = \phi$,
\item  $[B_{i,j}, B_{i,t} + (-1)^n B_{t,j}] =0$ for $ 1 \leq j < t < i
\leq k$, and
\item  $[B_{t,j}, B_{i,j} + B_{i,t}] =0$ for $ 1 \leq j < t < i \leq k$.
\end{enumerate}
\end{lem}

\vskip .2in
The relations in Lemma 4.2 above are called the (graded)
infinitesimal braid relations.
\vskip .2in

If n = 2m+2 with $ n > 2 $, the loop space of of a finite bouquet
of (n-1) spheres is homotopy equivalent to a product of loop
spaces of odd dimensional spheres by the Hilton-Milnor theorem.
Furthermore, the module of primitives for $\Omega S^{2k+1}$ is
given by a copy of the integers in degree 2k. Thus by the
Hilton-Milnor theorem, the module of primitives is given by the
Lie algebra generated by the $B_{i,j}$.

The Lie algebra generated by the $B_{i,j}$ is in the Hurewicz
image as the Samelson product of two elements x, and y in homotopy
map to the bracket $[\phi(x), \phi(y)]$ in homology where $\phi$
denotes the Hurewicz homomorphism. Thus the Hurewicz homomorphism
surjects to the module of primitives, and restricts to a
monomorphism on the torsion free summand of homotopy groups. The
kernel is precisely the torsion in the homotopy as the homology is
torsion free.

The theorem follows.

\vskip .2in

Remarks. There are analogous relations satisfied for the homology
of the loop spaces of many other configuration spaces of ordered
k tuples of points in certain manifolds M.

Define the "extended infinitesimal braid relations" as
follows:

\begin{enumerate}

\item  $[B_{i,j}, x_s] =0 $ if $ \{ i,j\} \cap \{
s \} = \phi$,
\item  $[B_{i,j}, x_i + x_j] =0$.
\end{enumerate}

\vskip .2in
These relations are satisfied in the homology of the loop space of the configuration
space based on the manifold M given by the product $R^1 \times N$. In this
case, the loop space splits as product where one factor is $(\Omega
N)^k$, and the classes $x_i$ above arise from a class in the
$i-th$ factor of $\Omega N$. Details are given in \cite {CG}.

\end{document}